\documentclass[12pt,a4paper]{article}
\usepackage{amsmath,amssymb,amsfonts,amsthm,graphicx,relsize,mathtools,mathrsfs,euscript,bbm,bbold,float}
\usepackage{enumitem,url,stmaryrd,indentfirst,microtype,xcolor,hyperref,breakurl,tikz,microtype}
\usepackage[labelsep=period, font=footnotesize, labelfont=bf]{caption}
\usepackage[english]{babel}
\usepackage[margin=0.9in]{geometry}
\DisableLigatures{encoding = *, family = *}

\usepackage{titlesec}
\titlelabel{\thetitle.\,}
\allowdisplaybreaks

\hypersetup{
    colorlinks=true,
    urlcolor=black,
    linkcolor=black,
    citecolor=black}

\makeatletter
\newtheorem*{rep@theorem}{\rep@title}
\newcommand{\newreptheorem}[2]{%
\newenvironment{rep#1}[1]{%
\def\rep@title{#2 \ref{##1}}%
\begin{rep@theorem}}%
{\end{rep@theorem}}}
\makeatother

\theoremstyle{plain}
\newtheorem{theorem}{Theorem}[section]
\newreptheorem{theorem}{Theorem}
\newtheorem{lemma}[theorem]{Lemma}

\theoremstyle{definition}
\newtheorem{definition}[theorem]{Definition}

\DeclareMathAlphabet{\mathbbmsl}{U}{bbm}{m}{sl}

\DeclareMathAlphabet{\mathpzc}{OT1}{pzc}{m}{it}

\newcommand{\mfrac}[2]{\mathchoice
{\vcenter{\hbox{\small$\displaystyle\frac{#1}{#2}$}}}
{\textstyle\frac{#1}{#2}}
{\scriptstyle\frac{#1}{#2}}
{\scriptscriptstyle\frac{#1}{#2}}}

\makeatletter
\DeclareRobustCommand\widecheck[1]{{\mathpalette\@widecheck{#1}}}
\def\@widecheck#1#2{%
\setbox\z@\hbox{\m@th$#1#2$}%
\setbox\tw@\hbox{\m@th$#1%
\widehat{%
\vrule\@width\z@\@height\ht\z@
\vrule\@height\z@\@width\wd\z@}$}%
\dp\tw@-\ht\z@
\@tempdima\ht\z@ \advance\@tempdima2\ht\tw@ \divide\@tempdima\thr@@
\setbox\tw@\hbox{%
\raise\@tempdima\hbox{\scalebox{1}[-1]{\lower\@tempdima\box
\tw@}}}%
{\ooalign{\box\tw@ \cr \box\z@}}}
\makeatother

\begin{document}

\title{\bf{\Large $\boldsymbol{2}$-Neighbor Bootstrap Percolation on Odd Graphs}\\[9mm]}

\author{
Ali  Mohammadian  \qquad Sina Rezaie Zareie  \qquad   Behruz  Tayfeh-Rezaie\\[4mm]
School of Mathematics,\\
Institute for Research in Fundamental Sciences (IPM),\\
P.O. Box 19395-5746, Tehran, Iran\\[4mm]
\href{mailto:ali\_m@ipm.ir}{ali\_m@ipm.ir}  \qquad
\href{mailto:rezaie@ipm.ir}{rezaie@ipm.ir} \qquad
\href{mailto:tayfeh-r@ipm.ir}{tayfeh-r@ipm.ir}\vspace{9mm}}

\date{}

\maketitle

\begin{abstract}
The $r$-neighbor bootstrap percolation process on a graph $G$ is a vertex-activation process that begins with a set of initially active vertices. In each subsequent round, every inactive vertex having at least $r$ active neighbors becomes active.
Denote  by $m(G,r)$  the minimum number of initially active vertices required for the activation process to spread to all remaining vertices of $G$.
In this article,   we consider the       odd graph  \(\mathbbmsl{O}_k\)  and the bipartite odd graph  \(\mathbbmsl{B}_k\), and  prove that $m(\mathbbmsl{O}_k, 2)=\lceil(k^2+3k+2)/4\rceil$ and $m(\mathbbmsl{B}_k, 2)=\lceil(k^2+2k+3)/2\rceil$.
The first formula confirms the quadratic-growth conjecture
proposed   in 2021  by Grippo, Pastine, Torres, Valencia-Pabon, and Vera.\\[3mm]
\end{abstract}

\section{Introduction}

Bootstrap percolation processes on graphs constitute a class of cellular automata, developed in the foundational work of von Neumann \cite{von}. One of the most extensively studied examples is the $r$-neighbor bootstrap percolation process  which was introduced in 1979 by Chalupa, Leath, and Reich \cite{cha}  in the context of ferromagnetism. This process is closely connected to several areas of graph theory,    including weak saturation and $P_3$-convexity. It has also appeared in the literature under various names  such as irreversible threshold process, influence propagation, and dynamic monopoly.

Throughout this article, all graphs are assumed to be  finite, undirected, and without multiple edges or loops. For a graph $G$, let $V(G)$ and $E(G)$ denote  the  vertex set and the  edge set of $G$, respectively.
Given a nonnegative integer $r$, the $r$-\textsl{neighbor bootstrap percolation process} on a graph $G$ begins with a set $V_0\subseteq V(G)$ of initially active vertices. For each integer $i\geqslant1$, the set of active vertices at round $i$ is defined by
\[
V_i=V_{i-1}\cup
\Big\{
v\in V(G)\,\Big|\,
\text{$v$ is adjacent to at least $r$ vertices in $V_{i-1}$}
\Big\}.
\]
The  set $V_0$ is called a \textsl{percolating set} if $V_t=V(G)$ for some integer $t\geqslant0$. We denote by $m(G,r)$ the minimum cardinality of a percolating set in the $r$-neighbor bootstrap percolation process on $G$.
Percolating sets are sometimes referred to as ``contagious sets''.
The case $r=2$ is particularly well studied in the literature.
In the context of graph convexity, the parameter $m(G,2)$ is called  the   ``$P_3$-hull number''    of $G$.

An edge analogue of the $r$-neighbor bootstrap percolation process is known as the    $r$-edge bootstrap percolation process. Given a nonnegative integer $r$, the $r$-\textsl{edge bootstrap percolation process} on a graph $G$ begins with a set $E_0\subseteq E(G)$ of initially active edges.  For each integer $i\geqslant1$, the set of active edges at round $i$ is defined by
\begin{equation}\label{defme}
E_i=E_{i-1}\cup
\Big\{
e\in E(G)\,\Big|\,
\text{at least one endpoint of $e$ is incident to at least $r$ edges in $E_{i-1}$}
\Big\}.
\end{equation}
The  set $E_0$ is called a \textsl{percolating set} if $E_t=E(G)$ for some integer $t\geqslant0$. Let $m_e(G,r)$ denote the minimum cardinality of a percolating set in the $r$-edge bootstrap percolation process on $G$.
The case $r=2$ was studied by Lenormand and Zarcone \cite{len} in 1984 under different terminology.

For every positive integer $r$, a result of Hambardzumyan, Hatami, and Qian \cite{ham} gives
\begin{equation}\label{m-me}
\mfrac{m_e(G,r)}{r}
\leqslant
m(G,r)
\leqslant
m_e(G,r)+
\Big|\big\{v\in V(G)\,\big|\,\deg(v)<r\big\}\Big|,
\end{equation}
where
$\deg(v)=|\{x\in V(G) \, | \,    x \text{ is adjacent to } v\}|$.
Determining the extremal parameters $m(G,r)$ and $m_e(G,r)$ is one of the fundamental problems in bootstrap percolation theory. These parameters have been examined for various families of graphs, including grids, tori, and Hamming graphs; see, for example, \cite{ham}, \cite{hamming}, and \cite{pete}. They have also been
investigated for  random graphs;   see,   for example,  \cite{random}.

Let $k$ be a positive integer and let $M$ be a  set of size $2k+1$.
The \textsl{odd graph} $\mathbbmsl{O}_k$ is the graph whose vertices are the $k$-subsets of $M$ in which  two vertices $S$ and $T$ are adjacent  if $S\cap T=\varnothing$.
The term ``odd graph'' is due to Biggs and Gardiner \cite{odd-man-out}  who observed that each edge can be labeled by the unique element of $M$ not contained in either
endpoint, representing the ``odd man out''. Odd graphs constitute a particularly interesting and highly symmetric subclass of Kneser graphs that exhibit remarkable algebraic and geometric properties not shared by other members of the Kneser family.

Recall that the \textsl{bipartite double cover} of a graph $G$ with vertex set $\{v_1,\ldots,v_n\}$ is the bipartite graph with vertex bipartition  $\{v_1',\ldots,v_n'\}\cup\{v_1'',\ldots,v_n''\}$  in which   $v_i'$ is adjacent to $v_j''$   if   $v_i$ is adjacent to $v_j$ in $G$. Clearly,   the bipartite double cover of $G$ is isomorphic to the direct product of $G$ and    the complete graph on two vertices. The  bipartite double cover  of   $\mathbbmsl{O}_k$ is called  the \textsl{bipartite odd graph} and is denoted by  $\mathbbmsl{B}_k$.
Bipartite odd graphs inherit many structural properties from odd graphs while exhibiting additional properties derived from their bipartite nature.
It is well-known that both odd graphs and bipartite odd graphs are distance-regular \cite{book}.
It is readily seen that
$\mathbbmsl{B}_k$ is isomorphic to the graph   whose vertex set consists of all $k$-subsets and all $(k+1)$-subsets of $M$  in which   two vertices $S$ and $T$ are adjacent  if $S\subsetneqq T$. We refer to this representation as the ``middle-level inclusion model'' of
$\mathbbmsl{B}_k$.

In this article, we concentrate   on the case where $r = 2$ and
study the parameters $m(\mathbbmsl{O}_k,2)$  and  $m(\mathbbmsl{B}_k,2)$. Specifically,   we   extend the polynomial method introduced by Hambardzumyan, Hatami, and Qian \cite{ham} to determine the exact values of these parameters, as stated in the following theorem.

\begin{theorem}\label{OO}
For every   positive   integer $k$,
\[m(\mathbbmsl{B}_k, 2)=\left\lceil\mfrac{k^2+2k+3}{2}\right\rceil \quad \text{ and } \quad m(\mathbbmsl{O}_k, 2)=\left\lceil\mfrac{k^2+3k+2}{4}\right\rceil.\]
\end{theorem}
\noindent In 2021,   Grippo, Pastine, Torres, Valencia-Pabon, and Vera \cite{hull} conjectured that $m(\mathbbmsl{O}_k,2)=\mathnormal{\Theta}(k^2)$. Theorem \ref{OO} resolves this conjecture affirmatively.

The article  is organized as follows. In Section \ref{pre}, we present the necessary preliminaries and introduce an alternative description of bipartite odd graphs.  In Sections \ref{BOG} and \ref{OG}, we    determine the exact values of   $m(\mathbbmsl{B}_k,2)$   and $m(\mathbbmsl{O}_k,2)$, respectively,  for all $k$.

\section{Preliminaries}\label{pre}

In this section, we provide the necessary preliminaries and background for the rest  of the article. We begin by presenting equivalent definitions for   bipartite odd graphs that are essential for our subsequent proofs. As  usual, for   two adjacent vertices $u$ and $v$, the notation  $uv$ or    \(u \sim v\)   represents  the edge joining them.

Let  $k$ be a positive integer and let $M=\{1,\ldots,2k+1\}$. We fix   the  partition $\{X, Y\}$ of $M$ with   $X=\{1,\ldots,k\}$ and $Y=\{k+1,\ldots,2k+1\}$.  For each subset $S\subseteq M$, we write   $S_X=S\cap X$ and $S_Y=S\cap Y$  to simplify the notation. Set
\[\mathcal{U}=\Big\{S\subseteq M \,\Big|\, |S_X|=|S_Y| \ \text{or}\ |S_Y|=|S_X|+1\Big\}.\]
The middle-level inclusion model of $\mathbbmsl{B}_k$ is isomorphic to the graph with vertex set $\mathcal{U}$ in which two vertices $S$ and $T$ are adjacent if $|S\triangle T|=1$, where $\triangle$ denotes the symmetric difference. To establish this isomorphism, consider the mapping  $\varphi:\mathcal{U}\to V(\mathbbmsl{B}_k)$ defined by
$\varphi(S)= (X\setminus S_X)\cup S_Y$.
Thus, we obtain an equivalent description of $\mathbbmsl{B}_k$. Figure \ref{B2} illustrates this description for the graph $\mathbbmsl{B}_2$.
Throughout the remainder of this article, we adopt this description   of $\mathbbmsl{B}_k$.
In particular, it follows immediately that $\mathbbmsl{B}_k$ is an induced subgraph of the $(2k+1)$-dimensional hypercube.

\begin{figure}[!h]
\centering
\begin{tikzpicture}[
    scale=1.2,
    vertex/.style={circle, fill=black, minimum size=4mm, inner sep=0pt},
    edge/.style={thick, black},
    lbl/.style={font=\small}
]

    \node[vertex] (R_p) at (-5, 0) {};
    \node[lbl, anchor=east] at (-5.1, 0) {$\varnothing$};

    \node[vertex] (A1_p) at (-3,  1.5) {}; \node[lbl, anchor=east] at (-2.73,  1.8) {$\mathsmaller{\{3\}}$};
    \node[vertex] (A2_p) at (-3,  0.0) {}; \node[lbl, anchor=east] at (-2.73,  0.3) {$\mathsmaller{\{4\}}$};
    \node[vertex] (A3_p) at (-3, -1.5) {}; \node[lbl, anchor=east] at (-2.73, -1.8) {$\mathsmaller{\{5\}}$};

    \node[vertex] (B1_p) at (-1,  2.5) {}; \node[lbl, anchor=east] at (-0.63,  2.8) {$\mathsmaller{\{1, 3\}}$};
    \node[vertex] (B2_p) at (-1,  1.5) {}; \node[lbl, anchor=east] at (-0.63,  1.8) {$\mathsmaller{\{2, 3\}}$};
    \node[vertex] (B3_p) at (-1,  0.5) {}; \node[lbl, anchor=east] at (-0.63,  0.2) {$\mathsmaller{\{1, 4\}}$};
    \node[vertex] (B4_p) at (-1, -0.5) {}; \node[lbl, anchor=east] at (-0.63, -0.8) {$\mathsmaller{\{2, 4\}}$};
    \node[vertex] (B5_p) at (-1, -1.5) {}; \node[lbl, anchor=east] at (-0.63, -1.8) {$\mathsmaller{\{1, 5\}}$};
    \node[vertex] (B6_p) at (-1, -2.5) {}; \node[lbl, anchor=east] at (-0.63, -2.8) {$\mathsmaller{\{2, 5\}}$};

    \node[vertex] (B1_dp) at (1,  2.5) {}; \node[lbl, anchor=west] at (0.5,  2.8) {$\mathsmaller{\{1, 3, 4\}}$};
    \node[vertex] (B2_dp) at (1,  1.5) {}; \node[lbl, anchor=west] at (0.5,  1.8) {$\mathsmaller{\{2, 3, 4\}}$};
    \node[vertex] (B3_dp) at (1,  0.5) {}; \node[lbl, anchor=west] at (0.5,  0.8) {$\mathsmaller{\{1,  4, 5\}}$};
    \node[vertex] (B4_dp) at (1, -0.5) {}; \node[lbl, anchor=west] at (0.5, -0.2) {$\mathsmaller{\{2, 4, 5\}}$};
    \node[vertex] (B5_dp) at (1, -1.5) {}; \node[lbl, anchor=west] at (0.6, -1.8) {$\mathsmaller{\{1, 3, 5\}}$};
    \node[vertex] (B6_dp) at (1, -2.5) {}; \node[lbl, anchor=west] at (0.5, -2.8) {$\mathsmaller{\{2, 3, 5\}}$};

    \node[vertex] (A1_dp) at (3,  1.5) {}; \node[lbl, anchor=west] at (2.6,  1.8) {$\mathsmaller{\{1, 2, 3, 4\}}$};
    \node[vertex] (A2_dp) at (3,  0.0) {}; \node[lbl, anchor=west] at (2.6,  0.3) {$\mathsmaller{\{1, 2, 4, 5\}}$};
    \node[vertex] (A3_dp) at (3, -1.5) {}; \node[lbl, anchor=west] at (2.6, -1.8) {$\mathsmaller{\{1, 2, 3, 5\}}$};

    \node[vertex] (R_dp) at (5, 0) {};
    \node[lbl, anchor=west] at (5.1, 0) {$\mathsmaller{\{1, 2, 3, 4,  5\}}$};

    \draw[edge] (R_p) -- (A1_p);
    \draw[edge] (R_p) -- (A2_p);
    \draw[edge] (R_p) -- (A3_p);

    \draw[edge] (R_dp) -- (A1_dp);
    \draw[edge] (R_dp) -- (A2_dp);
    \draw[edge] (R_dp) -- (A3_dp);

    \draw[edge] (A1_p) -- (B1_p); \draw[edge] (A1_p) -- (B2_p);
    \draw[edge] (A2_p) -- (B3_p); \draw[edge] (A2_p) -- (B4_p);
    \draw[edge] (A3_p) -- (B5_p); \draw[edge] (A3_p) -- (B6_p);

   \draw[edge] (A1_dp) -- (B1_dp);
   \draw[edge] (A1_dp) -- (B2_dp);
      \draw[edge] (A2_dp) -- (B3_dp);
   \draw[edge] (A2_dp) -- (B4_dp);
      \draw[edge] (A3_dp) -- (B5_dp);
   \draw[edge] (A3_dp) -- (B6_dp);

    \draw[edge] (B1_p) -- (B1_dp); \draw[edge] (B6_p) -- (B4_dp);
    \draw[edge] (B2_p) -- (B2_dp); \draw[edge] (B6_p) -- (B6_dp);
    \draw[edge] (B3_p) -- (B3_dp);
    \draw[edge] (B4_p) -- (B4_dp);

    \draw[edge] (B1_p) -- (B5_dp); \draw[edge] (B3_dp) -- (B5_p);
    \draw[edge] (B2_p) -- (B6_dp); \draw[edge] (B5_dp) -- (B5_p);
    \draw[edge] (B3_p) -- (B1_dp);
    \draw[edge] (B4_p) -- (B2_dp);

\end{tikzpicture}
\caption{An illustration of the graph $\mathbbmsl{B}_2$, which is also known as the Desargues graph in the literature.}\label{B2}
\end{figure}

Hambardzumyan, Hatami, and Qian \cite{ham} introduced a polynomial method for establishing lower bounds on $m_e(G,r)$ for every graph $G$ and every  nonnegative integer $r$.
Bidgoli, Mohammadian, and Tayfeh-Rezaie \cite{bid} slightly refined this method.
We  shall use    this refinement  in Section \ref{BOG}  to derive a lower bound on $m_e(\mathbbmsl{B}_k,2)$. It is worth noting that  Miralaei, Mohammadian, and Tayfeh-Rezaie \cite{hamming} showed that   the   polynomial   method  is a special case of a general linear-algebraic lemma due to Balogh, Bollob\'as, Morris, and Riordan \cite{linear}, which has been used to establish several results on weak saturation.  In what follows, we extend the polynomial method to develop a more general framework. We shall use this extension in Section \ref{OG}    to obtain a lower bound on $m_e(\mathbbmsl{O}_k,2)$.

\begin{definition}\label{defhh}
Let $\mathbbmsl{F}$ be a field,   $r$ be a nonnegative integer,
and   $G$ be a graph equipped with a proper edge-coloring
$\varepsilon:E(G)\rightarrow\mathbbmsl{F}$. Also,
for every incident vertex-edge pair $(v,e)$, fix a nonzero  weight
$\omega_{v,e}\in\mathbbmsl{F}\setminus\{0\}$.
Define $W_{\varepsilon,\omega}(G,r)$ to be the vector space over
$\mathbbmsl{F}$ consisting of all functions
$\varphi:E(G)\rightarrow\mathbbmsl{F}$ for which there exists
a family of polynomials $\{P_v(x)\}_{v\in V(G)}$ satisfying
\begin{itemize}[noitemsep,topsep=0pt]
\item[{\rm (i)}]
$\deg P_v(x)\leqslant r-1$ for every $v\in V(G)$;
\item[{\rm (ii)}] $\omega_{u,e}P_u(\varepsilon(e))=\omega_{v,e}P_v(\varepsilon(e))=\varphi(e)$ for every edge $e=uv\in E(G)$.
\end{itemize}
We say that $\{P_v(x)\}_{v\in V(G)}$ recognizes $\varphi$.
We adopt the convention that the degree of the zero polynomial is $-1$.
\end{definition}

The next theorem provides a linear-algebraic lower bound on $m_e(G,r)$ for   every graph $G$ and every  nonnegative integer $r$, extending Theorem 2 of \cite{ham}. This result plays a crucial role in our proofs.

\begin{theorem}\label{hh}
Let $G$ be a graph, $r$ be  a nonnegative integer, and $\mathbbmsl{F}$ be  a field. Also, let $\varepsilon:E(G)\rightarrow\mathbbmsl{F}$ be a proper edge-coloring  and suppose that a nonzero weight $\omega_{v,e}\in\mathbbmsl{F}\setminus\{0\}$ is assigned to every incident vertex-edge pair $(v,e)$. Then,
\[
m_e(G,r)\geqslant \dim_{\mathbbmsl{F}}W_{\varepsilon,\omega}(G,r).
\]
\end{theorem}

\begin{proof}
If $r=0$, then $W_{\varepsilon,\omega}(G,r)=\{0\}$ and so
the asserted inequality holds trivially. We may thus  assume that
$r\geqslant 1$ and let $E_0\subseteq E(G)$ be a percolating set of minimum possible size for the $r$-edge bootstrap percolation process on $G$. For every  integer $i\geqslant1$, let $E_i$ be the   subset of $E(G)$ defined in  \eqref{defme}.
Consider the  restriction  map
$\pi:W_{\varepsilon,\omega}(G,r)\rightarrow\mathbbmsl{F}^{E_0}$ defined  by
$\pi(\varphi)=\varphi_{\raisebox{0.15ex}{$\scriptstyle|_{E_0}$}}$ for every  function $\varphi\in W_{\varepsilon,\omega}(G,r)$. Clearly, $\pi$ is linear.
We prove   that $\pi$ is injective.

Let $\varphi\in\ker\pi$  and fix a recognizing
family $\{P_v(x)\}_{v\in V(G)}$ for $\varphi$.
We claim   that $\varphi(e)=0$ for every edge $e\in E_i$ and every integer  $i\geqslant 0$. We establish this  claim by induction on $i$. The  claim holds for  $i=0$, as  $\varphi\in\ker\pi$.
So, suppose that $i\geqslant1$ and  the claim holds for $i-1$. Let
$e\in E_i\setminus E_{i-1}$.
By   the activation rule, some endpoint $v$ of $e$ is
incident with  $r$ distinct edges
$e_1,\ldots,e_r\in E_{i-1}$.
We find from   the induction hypothesis that  $\varphi(e_1)=\cdots=\varphi(e_r)=0$.   By the definition of a recognizing family, we have
$P_v(\varepsilon(e_\ell))
=\omega_{v,e_\ell}^{-1}\varphi(e_\ell)
=0$ for $\ell=1, \ldots, r$.
Since  $\varepsilon$ is a proper edge-coloring and the edges
$e_1,\ldots,e_r$ are incident with  the same vertex,
the  elements
$\varepsilon(e_1),\ldots,\varepsilon(e_r)$
are pairwise distinct.
Consequently,   $P_v$ has at least $r$ distinct roots in $\mathbbmsl{F}$.
It follows from
$\deg P_v\leqslant r-1$    that $P_v=0$. This  implies  that $\varphi(e)=\omega_{v,e}P_v(\varepsilon(e))=0$, completing  the induction and establishing  the claim.

Since $E_0$ percolates, there exists an integer  $t\geqslant0$
such that $E_t=E(G)$.
We have shown above  that   $\varphi$ vanishes on $E_t$ and hence
$\varphi=0$.  It follows  that $\ker\pi=\{0\}$  and so   $\pi$ is injective. Therefore,
\[
\dim_{\mathbbmsl{F}}W_{\varepsilon,\omega}(G,r)
\leqslant \dim_{\mathbbmsl{F}}\mathbbmsl{F}^{E_0}
=|E_0|=m_e(G,r).
\qedhere\]
\end{proof}

\section{Bipartite odd graphs}\label{BOG}

In this section, we determine the exact value of $m(\mathbbmsl{B}_k,2)$,
as stated in Theorem \ref{OO}. We begin by establishing a lower bound
for $m_e(\mathbbmsl{B}_k,2)$ using the polynomial method described in
Section \ref{pre}. We then derive an upper bound via an explicit
construction of a percolating set for the $2$-neighbor bootstrap
percolation process on $\mathbbmsl{B}_k$. Combining these bounds with
\eqref{m-me} yields the desired exact value.

\begin{theorem}\label{lowerboundmeBOk}
For every positive integer \(k\),
$$m_e(\mathbbmsl{B}_k,2)\geqslant k^2+2k+3.$$
\end{theorem}

\begin{proof}
Fix a positive integer $k$. We emphasize that $\mathbbmsl{B}_k$ is
understood according to the definition given in Section \ref{pre}.
For any two adjacent vertices $S,T\in V(\mathbbmsl{B}_k)$, we assign
to the edge $S\sim T$ the unique element of $S\triangle T$.
This assignment clearly defines a proper edge-coloring of
$\mathbbmsl{B}_k$.
To apply Theorem \ref{hh}, let $\varepsilon$ denote the proper
edge-coloring defined above  and take $\omega_{v,e}=1$ for every incident vertex-edge pair $(v,e)$. We also regard
$W_{\varepsilon,\omega}(\mathbbmsl{B}_k,2)$ as a vector space over
$\mathbbmsl{R}$.
Now, in view of Theorem \ref{hh}, it suffices to exhibit
$k^2+2k+3$ linearly independent functions in
$W_{\varepsilon,\omega}(\mathbbmsl{B}_k,2)$.
Consider  two polynomials
$I_S(x)=1$ and $J_S(x)=x$ for every  vertex  $S\in V(\mathbbmsl{B}_k)$. By  Definition \ref{defhh}, it is clear  that
$\{I_S(x)\}_{S\in V(\mathbbmsl{B}_k)}$ and $\{J_S(x)\}_{S\in V(\mathbbmsl{B}_k)}$ recognize   two functions    $\phi, \psi\in W_{\varepsilon,\omega}(\mathbbmsl{B}_k,2)$,  respectively.
Now, for  any    element \(b\in Y\) and  any   vertex     $S\in  V(\mathbbmsl{B}_k)$, define  the polynomial
\[P_S^b(x)=
\begin{cases}
x-b &   \text{ if } b\in S,\\[2mm]
0 &   \text{ otherwise}.
\end{cases}\]
For any     element \(b\in Y\) and   any  edge $S\sim T$ of  $\mathbbmsl{B}_k$ with    color  $c$, we   claim   that     $P_S^b(c)=P_T^b(c)$.     If either $b\in S\cap T$ or $b\notin S\cup T$, then    $P_S^b(x)=P_T^b(x)$.    Otherwise,   $b=c$ and so  $P_S^b(c)=P_T^b(c)=0$, proving the claim.
From   Definition \ref{defhh},
$\{P^b_S(x)\}_{S\in V(\mathbbmsl{B}_k)}$  recognizes    a  function     $ {\phi}_b\in W_{\varepsilon,\omega}(\mathbbmsl{B}_k,2)$ for each    element \(b\in Y\).
Finally, for any
two elements  \(a\in X, b\in Y\) and any   vertex     $S\in  V(\mathbbmsl{B}_k)$, define  the polynomial
\[Q_S^{a, b}(x)=
\begin{cases}
\mathnormal{\Gamma}_S^{a, b}\big(x-a\big)  & \text{ if }  |S_X|=|S_Y| \text{ and } a,b\in S,\\[2mm]
\mathnormal{\Gamma}_S^{a, b}\big(x-b\big)  & \text{ if }  |S_Y|=|S_X|+1 \text{ and } a,b\in S,\\[2mm]
0  & \text{ otherwise},
\end{cases}\]
where
\[\mathnormal{\Gamma}_S^{a, b}
=\prod_{u\in S_X\setminus\{a\}}
\mfrac{u-b}{u-a}
\prod_{v\in S_Y\setminus\{b\}}
\mfrac{v-a}{v-b}.\]
For  any  two     elements  \(a\in X, b\in Y\)   and   any  edge $S\sim T$ of  $\mathbbmsl{B}_k$ with  color  $c$, we   claim   that     $Q_S^{a, b}(c)=Q_T^{a, b}(c)$.  Without loss of generality, we may assume that  $c\in T\setminus S$.   If \(\{a,b\}\nsubseteq T\), then \(\{a,b\}\nsubseteq S\)   and  hence
$Q_S^{a, b}(x)=Q_T^{a, b}(x)=0$. So,  we consider the case   \(\{a,b\}\subseteq T\). First, suppose that \(\{a,b\}\subseteq S\). This forces  that $c\notin\{a, b\}$. If $c\in X$, then \(|S_Y|=|S_X|+1\), \(|T_X|=|T_Y|\),   and
\[\mathnormal{\Gamma}_T^{a, b}
=
\mathnormal{\Gamma}_S^{a, b} \, \, \mfrac{c-b}{c-a},\]
which implies  that
$Q_S^{a, b}(c)=Q_T^{a, b}(c)$.
If $c\in Y$, then \(|S_X|=|S_Y|\), \(|T_Y|=|T_X|+1\),   and
\[\mathnormal{\Gamma}_T^{a, b}
=
\mathnormal{\Gamma}_S^{a, b} \, \, \mfrac{c-a}{c-b},\]
which yields  that
$Q_S^{a, b}(c)=Q_T^{a, b}(c)$.
Next, suppose that \(\{a,b\}\nsubseteq S\). This means that   $Q_S^{a, b}(x)=0$ and   forces  that $c\in\{a, b\}$. If $a=c$, then it follows from  \(a\in X\)   that $|T_X|=|T_Y|$  and      $Q_T^{a, b}(c)=0$.
If $b=c$, then  one  deduces  from    \(b\in Y\)   that $|T_Y|=|T_X|+1$  and     $Q_T^{a, b}(c)=0$. This shows that  $Q_S^{a, b}(c)=Q_T^{a, b}(c)=0$, proving   the claim. In view of    Definition \ref{defhh},
$\{Q^{a, b}_S(x)\}_{S\in V(\mathbbmsl{B}_k)}$  recognizes    a  function     $ {\psi}_{a, b}\in W_{\varepsilon,\omega}(\mathbbmsl{B}_k,2)$ for any  two    elements   \(a\in X\) and \(b\in Y\).

It remains  to establish  that   $\{\phi, \psi\}\cup \{{\phi}_b\}_{b\in Y}\cup\{{\psi}_{a, b}\}_{a\in X, b\in Y}$ is  linearly independent.
To this end,  suppose  that
\begin{equation}\label{2sigma}
\lambda\phi+\mu\psi+\sum_{b\in Y}\lambda_{b} {\phi}_b+\sum_{a\in X, b\in Y}\mu_{a, b} {\psi}_{a, b}=0
\end{equation}
for some scalars   $\lambda, \mu, \lambda_{b}, \mu_{a, b} \in\mathbbmsl{R}$.
We show that all of these scalars  vanish.
Fix an   element $a_0\in X$ and
two distinct elements $b_0, b'_0\in Y$. Since all functions in   $\{{\phi}_b\}_{b\in Y}\cup\{{\psi}_{a, b}\}_{a\in X, b\in Y}$ vanish on the edges $\varnothing \sim \{b_0\}$ and  $\varnothing \sim \{b'_0\}$, it follows from \eqref{2sigma} that
$\lambda+\mu b_0=0$ and
$\lambda+\mu b'_0=0$,
yielding $\lambda=\mu =0$.
Consequently, \eqref{2sigma} can be rewritten as
\begin{equation}\label{2sigma2}
\sum_{b\in Y}\lambda_{b} {\phi}_b+\sum_{a\in X, b\in Y}\mu_{a, b} {\psi}_{a, b}=0.
\end{equation}
Denote by $e$ and $e'$   the edges  $\{b_0\}  \sim \{a_0, b_0\}$  and  $\{a_0, b_0\}  \sim \{a_0, b_0, b'_0\}$, respectively. For every    element \(b\in Y\),
we have
\[{\phi}_b(e)=
\begin{cases}
a_0-b_0 &  \text{ if }  b=b_0,\\[2mm]
0 &   \text{ otherwise}.
\end{cases}\]
Since all functions in   $\{{\psi}_{a, b}\}_{a\in X, b\in Y}$ vanish on   $e$, we  deduce  from   \eqref{2sigma2} that $\lambda_{b_0}=0$, meaning that   $\lambda_{b}=0$ for all     elements \(b\in Y\). Consequently, we may rewrite \eqref{2sigma2}  as
\begin{equation}\label{2sigma22}
\sum_{a\in X, b\in Y}\mu_{a, b} {\psi}_{a, b}=0.
\end{equation}
For any    elements \(a\in X\) and  \(b\in Y\), we have
\[\psi_{a, b}(e')=
\begin{cases}
b'_0-a_0 & \text{ if }   a=a_0 \text{ and } b=b_0,\\[2mm]
0 &    \text{ otherwise}.
\end{cases}\]
Combining this observation with \eqref{2sigma22}, we conclude that
$\mu_{a_0,b_0}=0$. It follows that   $\mu_{a,b}=0$ for all elements
$a\in X$ and $b\in Y$.
\end{proof}

The following key lemma is needed to establish the main result of this section.

\begin{lemma}\label{pairactive}
Let $k$ be an integer at least $2$. Assume that
the initial activation of a subset     $\mathcal{W}\subseteq V(\mathbbmsl{B}_k)$      eventually activates all  size-two   vertices      in  the   $2$-neighbor bootstrap percolation process on   $\mathbbmsl{B}_k$. Then,    $\mathcal{W}$  is   a percolating set.
\end{lemma}

\begin{proof}
For each integer \(\ell \geqslant 0\), let \(\mathcal{V}_\ell\) be  the set of vertices at distance \(\ell\) from the vertex  $\varnothing$.
By the description of $\mathbbmsl{B}_k$ given in Section \ref{pre}, the set
\(\mathcal{V}_\ell\) consists of all  vertices of cardinality  \(\ell\). Furthermore,
it is
clear   by the description  that every vertex in \(\mathcal{V}_1\) has exactly $k$  neighbors in \(\mathcal{V}_2\)  and  every vertex in \(\mathcal{V}_\ell\) has at least two   neighbors in \(\mathcal{V}_{\ell-1}\) whenever \(\ell \geqslant 3\).
Since  the initial activation of    $\mathcal{W}$    eventually activates all vertices in \(\mathcal{V}_2\),   it follows  that every vertex in
\(\mathcal{V}_\ell\) will eventually become active  for all \(\ell \geqslant 1\).
Finally,  because the vertex  \(\varnothing\) has exactly \(k+1\) active neighbors in \(\mathcal{V}_1\),   it must also become active. This demonstrates that  $\mathcal{W}$  is   a percolating set.
\end{proof}

We are now ready to prove the first assertion of Theorem \ref{OO}.

\begin{theorem}\label{mBOk}
For every positive integer \(k\),
\[m(\mathbbmsl{B}_k,2)=
\left\lceil\mfrac{k^2+2k+3}{2}\right\rceil.\]
\end{theorem}

\begin{proof}
The case $k=1$ is immediate, as
$\mathbbmsl{B}_1$ is a cycle graph on six vertices.
So, we assume henceforth that $k\geqslant 2$.
Combining \eqref{m-me}  and Theorem
\ref{lowerboundmeBOk}, we obtain
$
m(\mathbbmsl{B}_k,2)
\geqslant
(k^2+2k+3)/2$.
Also,    by   Lemma \ref{pairactive}, it suffices to construct a subset   $\mathcal{W}\subseteq V(\mathbbmsl{B}_k)$ of size
$\lceil(k^2+2k+3)/2\rceil$ such that the initial activation of the vertices in    $\mathcal{W}$      eventually activates all size-two vertices. We proceed by considering two cases based on the parity of \(k\).

\noindent\textbf{Case 1.} \(k\) is odd.

Define
\begin{align*}
\mathcal{P}_1
&=
\Big\{\{k+1\},\{k+2\}\Big\},\\
\mathcal{P}_2
&=
\Bigl\{\{1,j\} \, \Big| \,  k+3\leqslant j\leqslant 2k+1\Bigr\},\\
\mathcal{P}_3
&=
\Bigl\{\{i,k+1,k+2\} \, \Big| \, 1\leqslant i\leqslant k\Bigr\}, \\
\mathcal{P}_4
&=
\Big\{
\{i,2j,2j+1\} \, \Big| \,
2\leqslant i\leqslant k    \text{ and }     \mfrac{k+3}{2}\leqslant j\leqslant k
\Big\},
\end{align*}
and let
$\mathcal{W}
=
\mathcal{P}_1\cup
\cdots\cup
\mathcal{P}_4$. An easy  computation reveals that
$|\mathcal{W}|
=
(k^2+2k+3)/2$.
So, it remains to show    that the initial activation of the vertices in    $\mathcal{W}$    eventually activates all  size-two  vertices.  As     \(\varnothing\) is adjacent to both of the  active  vertices
\(\{k+1\}, \{k+2\}\in \mathcal{P}_1\), the vertex
\(\varnothing\)   becomes active. For any  integer     $j\in\{k+3, \ldots, 2k+1\}$, the vertex    $\{j\}$
becomes active because it has two active  neighbors
$\varnothing$ and $\{1, j\}\in \mathcal{P}_2$.
Finally, for all integers   $i\in\{1, \dots,  k\}$ and $j\in\{k+1, k+2\}$, the vertex
\(\{i,j\}\) becomes active
due to its adjacency to   both of the active   vertices $\{j\}\in \mathcal{P}_1$ and
$\{i,k+1,k+2\}\in \mathcal{P}_3$.

We next show that every remaining  size-two vertex  becomes active.
Let \(\{i, j\}\) be a vertex with
$2\leqslant i\leqslant k$ and
$k+3\leqslant j\leqslant 2k+1$.
Since  \(\{i, j\}\)  is adjacent to both of the  active  vertices $\{j\}$ and $\{i, j, \smash{\widecheck{j}}\}\in\mathcal{P}_4$, where
\begin{equation}\label{jprime}
\widecheck{j}=
\begin{cases}
j-1 & \text{ if } j \text{ is odd},\\[2mm]
j+1 & \text{ otherwise},
\end{cases}\end{equation}
we conclude that the vertex
\(\{i,j\}\) becomes active.

\noindent\textbf{Case 2.} \(k\) is even.

Define
\begin{align*}
\mathcal{Q}_1
&=
\Bigl\{\{k+1\},\{k+2\},\{k+3\}\Bigr\},\\
\mathcal{Q}_2
&=
\Bigl\{\{1,j\} \, \Big| \,  k+4\leqslant j\leqslant 2k+1\Bigr\},\\
\mathcal{Q}_3
&=
\Bigl\{\{i,k+2,k+3\} \, \Big| \,  1\leqslant i\leqslant k\Bigr\},\\
\mathcal{Q}_4
&=
\Bigl\{
\{i,2j,2j+1\} \, \Big| \,
2\leqslant i\leqslant k  \text{ and }
\mfrac{k+4}{2}\leqslant j\leqslant k \Bigr\},\\
\mathcal{Q}_5
&=
\Bigl\{
\{2i-1,2i,k+1,k+2\} \, \Big|  \,
1\leqslant i\leqslant \mfrac{k}{2}
\Bigr\},
\end{align*}
and set
$\mathcal{W}
=
\mathcal{Q}_1\cup
\cdots
\cup
\mathcal{Q}_5$.
A straightforward calculation  shows that
$|\mathcal{W}|
=
(k^2+2k+4)/2=\lceil(k^2+2k+3)/2\rceil$.
Thus,   it is enough  to demonstrate  that the initial activation of the vertices in    $\mathcal{W}$    eventually activates all  size-two  vertices.
Since   \(\varnothing\) is adjacent to both of the active   vertices
\(\{k+1\}, \{k+2\}\in \mathcal{Q}_1\), the vertex
\(\varnothing\)     becomes  active.
For any  integer     $j\in\{k+4, \ldots, 2k+1\}$,  the  vertex \(\{j\}\) becomes active because   it  possesses  the two active  neighbors $\varnothing$ and $\{1, j\}\in \mathcal{Q}_2$.
For all integers   $i\in\{1, \dots,  k\}$ and $j\in\{k+2, k+3\}$, the vertex    $\{i,j\}$    becomes active
due to its adjacency to   both of the active   vertices
$\{j\}\in \mathcal{Q}_1$ and
$\{i,k+2,k+3\}\in \mathcal{Q}_3$.
For  all integers  \( i\in\{1, \ldots,  k/2\}\) and $i'\in\{2i-1, 2i\}$,
the vertex
$\{i',k+1,k+2\}$    becomes active because it has two active  neighbors, namely
$\{i',k+2\}$ and $\{2i-1, 2i,k+1,k+2\} \in \mathcal{Q}_5$.
Hence, the vertices $\{i,k+1,k+2\}$ have been activated
for $i=1,\ldots,k$.
Finally, for  any  integer   \( i\in\{1, \ldots,  k\}\),
since  $\{i,k+1\}$   is adjacent to both of the  active   vertices
$\{k+1\}\in\mathcal{Q}_1$ and $\{i,k+1,k+2\}$, the vertex
$\{i,k+1\}$    becomes active.

We next show    that all remaining  size-two vertices      become active.
Let \(\{i, j\}\) be a vertex with
\(2\leqslant i\leqslant k\) and
\(k+4\leqslant j\leqslant 2k+1.\)
As  \(\{i, j\}\)  is adjacent to both of the active   vertices $\{j\}$ and $\{i, j, \smash{\widecheck{j}}\}\in\mathcal{Q}_4$,  where $\smash{\widecheck{j}}$ is introduced in \eqref{jprime},
we conclude that the vertex
\(\{i,j\}\) becomes active.
\end{proof}

\section{Odd graphs}\label{OG}

In this section,  we determine   the exact value  of   $m(\mathbbmsl{O}_k,2)$, as stated in Theorem \ref{OO}. We first   derive   the  lower bound $m(\mathbbmsl{O}_k, 2)\geqslant(k^2+3k+2)/4$  by applying  the polynomial method described  in Section \ref{pre}   directly to \(\mathbbmsl{O}_k\). We then establish     a matching upper bound by explicitly constructing a percolating set for   the    $2$-neighbor  bootstrap percolation process on $\mathbbmsl{O}_k$. Together, these bounds yield the exact value of   $m(\mathbbmsl{O}_k,2)$.
Before proceeding with the proofs, we recall  a well-known fact about polynomials that underlies the Lagrange interpolation method.
We omit the proof,   as the assertion  follows directly from a one-line  argument.

\begin{lemma}\label{Lag}
Let $n$  be a positive  integer, $\mathbbmsl{F}$ be a field, and
$a_1,\ldots,a_n$ be distinct elements of $\mathbbmsl{F}$. For each integer $i\in\{1,\ldots,n\}$,
define
\[
p_i(x)=\prod_{\substack{1\leqslant\ell\leqslant  n\\ \ell \neq i}}(x-a_\ell).
\]
Then, for every polynomial $p(x)$   over $\mathbbmsl{F}$ of degree
at most $n-1$,
\[
p(x)=\sum_{i=1}^{n}
p(a_i)p_i(a_i)^{-1}p_i(x).
\]
\end{lemma}

\begin{theorem}\label{o-g-l-b}
For every positive integer \(k\),
\[m_e(\mathbbmsl{O}_k,2)\geqslant \binom{k+2}{2}.\]
\end{theorem}

\begin{proof}
Fix a positive integer  $k$     and identify the vertices of
$\mathbbmsl{O}_k$ with the $k$-subsets of the ground set
$M=\{1,\ldots,2k+1\}$.
Recall that two vertices in $\mathbbmsl{O}_k$ are adjacent if and only
if they are disjoint. For any two adjacent vertices
$S,T\in V(\mathbbmsl{O}_k)$, we assign to the edge $S\sim T$ the
unique element of $M\setminus(S\cup T)$. This assignment clearly
defines a proper edge-coloring of $\mathbbmsl{O}_k$.
To apply Theorem \ref{hh}, let $\varepsilon$ denote the proper
edge-coloring defined above  and assume that $\omega$ is a
nowhere-zero weighting function satisfying
$\omega_{u,e}=-\omega_{v,e}$ for every edge
$e=uv\in E(\mathbbmsl{O}_k)$. We also regard
$W_{\varepsilon,\omega}(\mathbbmsl{O}_k,2)$ as a vector space over
$\mathbbmsl{R}$.
Now, in view of Theorem \ref{hh}, it suffices to construct
$\binom{k+2}{2}$ linearly independent functions in
$W_{\varepsilon,\omega}(\mathbbmsl{O}_k,2)$.
For each subset  $S\subseteq  M$, define
\[
f_S( t )=\prod_{a\in S}( t -a).
\]
For  any nonnegative integers $i$ and $j$  and any   polynomials $p( t )=\sum_\ell p_\ell  t ^\ell$ and  $q(  t )=\sum_\ell q_\ell  t ^\ell$,  define
\begin{equation}\label{alt}\mathnormal{\Gamma}_{ij}\big(p( t ), q( t )\big)=p_iq_j-p_jq_i.\end{equation}
Now, define
\[
P_S^{ij}(x)=\mathnormal{\Gamma}_{ij}\big(( t -x)f_S( t ),f_{M\setminus S}( t )\big).
\]
Then,  $P_S^{ij}(x)$ is a polynomial in $x$ of degree at most $1$.
For any  nonnegative    integers   $i, j$  and  every    edge $S\sim T$ of  $\mathbbmsl{O}_k$ with    color  $c$, it follows  from $M=S\cup T\cup \{c\}$  that      $$P_S^{ij}(c)=\mathnormal{\Gamma}_{ij}
\big(( t -c)f_S( t ),f_{M\setminus S}( t )\big)=\mathnormal{\Gamma}_{ij}
\big(f_{M\setminus T}( t ),( t -c)f_T( t )\big)=-P_T^{ij}(c).$$
Thus,    by     Definition \ref{defhh},  the family
$\{P_S^{ij}(x)\}_{S\in V(\mathbbmsl{O}_k)}$  recognizes    a  function     $ {\varphi}_{ij}\in W_{\varepsilon,\omega}(\mathbbmsl{O}_k,2)$.
So, it suffices     to establish  that   $\{\varphi_{ij}\}_{0\leqslant i<j\leqslant k+1}$ is  linearly independent.
To this end,  suppose  that
\begin{equation}\label{2lambda}
\sum_{0\leqslant i<j\leqslant k+1}\lambda_{ij}\varphi_{ij}=0
\end{equation}
for some scalars   $\lambda_{ij}  \in\mathbbmsl{R}$. We show that all of these scalars  vanish.
It is straightforward to verify that, for any indices $i, i_0, j, j_0$ satisfying $i<j$ and $i_0<j_0$,
\[
\mathnormal{\Gamma}_{ij}\left(t^{i_0},t^{j_0}\right)=
\begin{cases}
1 & \text{ if } i=i_0 \text{ and } j=j_0,\\[2mm]
0 & \text{ otherwise}.
\end{cases}
\]
Setting
\[
\mathnormal{\Gamma}=\sum_{0\leqslant i<j\leqslant k+1}\lambda_{ij}\mathnormal{\Gamma}_{ij},
\]
we derive   that $\mathnormal{\Gamma}(t^{i_0},t^{j_0})=\lambda_{i_0j_0}$  for any    indices  $i_0,  j_0$ with  $i_0<j_0$. So,  it is enough   to show that $\mathnormal{\Gamma}(t^i,t^j)=0$   for all indices   $i,  j$ with  $i<j$.

By     Definition \ref{defhh}, for every vertex $S\in V(\mathbbmsl{O}_k)$ and every element  $c\in M\setminus  S$, applying  both sides of    \eqref{2lambda}  to the edge  $S\sim M\setminus(S\cup\{c\})$ yields
$\mathnormal{\Gamma} (( t -c)f_S( t ),f_{M\setminus S}( t ) )=0$, meaning that
\begin{equation}\label{3lambda}
-\mathnormal{\Gamma}\big(f_S( t ),f_{M\setminus S}( t )\big)c+\mathnormal{\Gamma}\big(  t f_S( t ),f_{M\setminus S}( t )\big)=0.
\end{equation}
Since $|M\setminus  S|=k+1\geqslant2$, there are at least two choices for  $c$   in   \eqref{3lambda},  so both coefficients in this affine expression vanish.
In particular,
\begin{equation}\label{4lambda}\mathnormal{\Gamma}\big(tf_S( t ),f_{M\setminus S}( t )\big)=0\end{equation} for each  vertex $S\in V(\mathbbmsl{O}_k)$.
We claim that
\begin{equation}\label{iijj} \mathnormal{\Gamma}\big( t^2  p( t ), q( t )\big)=\mathnormal{\Gamma}\big( t p( t ),  t  q( t )\big)\end{equation}
for  any     polynomials $p( t )$ and  $q(  t )$ of degrees at most  $k-1$ and $k$, respectively.
For  each  $(k-1)$-subset $A\subseteq M$   and  each  polynomial $q( t )$ of degree at most $k$,   Lemma \ref{Lag} guarantees   the existence of scalars   $\{ \beta_s \}_{s \in M\setminus A}$ such that $q( t )=\sum_{s \in M\setminus A} \beta_s  f_{M\setminus(A\cup\{s \})}( t )$ and $t  q( t )=\sum_{s \in M\setminus A}s  \beta_s  f_{M\setminus(A\cup\{s \})}( t )$.  Now, we may write
\begin{align}
\mathnormal{\Gamma}\big( t^2   f_A( t ), q( t )\big)\nonumber
&=\sum_{s \in M\setminus A} \beta_s  \mathnormal{\Gamma}\big( t^2   f_A( t ), f_{M\setminus(A\cup\{s \})}( t )\big) \\\nonumber
&=\sum_{s \in M\setminus A} \beta_s  \Big(\mathnormal{\Gamma}\big(tf_{A\cup\{s \}}( t ), f_{M\setminus(A\cup\{s \})}( t )\big)+\mathnormal{\Gamma}\big(st  f_A( t ), f_{M\setminus(A\cup\{s \})}( t )\big)\Big) \\\nonumber
&\stackrel{\mathclap{\text{by } \eqref{4lambda}}}{=}\sum_{s \in M\setminus A}s  \beta_s \mathnormal{\Gamma}\big(tf_A( t ), f_{M\setminus(A\cup\{s \})}( t )\big) \\
&=\mathnormal{\Gamma}\big( t f_A( t ),  t  q( t )\big).\label{ssxx}
\end{align}
For each  polynomial $p( t )$ of degree at most $k-1$,  it follows from Lemma \ref{Lag}  that there exist scalars     $\{\alpha_s \}_{s \in X}$ such that $p( t )=\sum_{s \in X}\alpha_s  f_{X\setminus\{s \}}( t )$.
Now,  for each  polynomial $q( t )$ of degree at most $k$,  we have
\begin{align*}
\mathnormal{\Gamma}\big( t^2  p( t ), q( t )\big)
&=\sum_{s \in X} \alpha_s  \mathnormal{\Gamma}\big( t^2   f_{X\setminus\{s \}}( t ), q( t )\big) \\
&\stackrel{\mathclap{\text{by } \eqref{ssxx}}}{=}  \, \, \, \,     \sum_{s \in X} \alpha_s  \mathnormal{\Gamma}\big(  t   f_{X\setminus\{s \}}( t ), tq( t )\big) \\
&=\mathnormal{\Gamma}\big( t p( t ),  t  q( t )\big),
\end{align*}
establishing the claim.
By  \eqref{iijj},    we obtain     for any  indices  $i,  j$ with  $1\leqslant i<j\leqslant k+1$ that
\begin{equation}\label{1iijj}\mathnormal{\Gamma}\big(  t^i,  t^j \big)=\mathnormal{\Gamma}\big(  t^{i+1},  t^{j-1} \big)=\mathnormal{\Gamma}\big(  t^{i+2},  t^{j-2} \big)=\cdots=\mathnormal{\Gamma}\big(  t^j,  t^i \big).\end{equation}
On the other hand, it follows directly from \eqref{alt} that
\begin{equation}\label{alternation}\mathnormal{\Gamma}\big(   p( t ), p( t )\big)=0 \quad \text{and} \quad \mathnormal{\Gamma}\big(   p( t ), q( t )\big)=-\mathnormal{\Gamma}\big(      q( t ),   p( t )\big)\end{equation}
for any   polynomials $p( t )$ and  $q(  t )$ of   degrees at most $k+1$.
In particular,
\(\mathnormal{\Gamma}(t^i,t^j)=-\mathnormal{\Gamma}(t^j,t^i)\). Combining this identity with \eqref{1iijj}, we conclude that       \(\mathnormal{\Gamma}(t^i,t^j)=0\)  whenever   $1\leqslant i<j\leqslant k+1$.
Moreover,   by applying \eqref{iijj} together with
\eqref{alternation}, we obtain
$$\mathnormal{\Gamma}\big(  1,  t^j \big)=-\mathnormal{\Gamma}\big(  t^j,  1 \big)=-\mathnormal{\Gamma}\big(  t^{j-1},  t  \big)=\mathnormal{\Gamma}\big(  t ,  t^{j-1} \big)=0$$
for each integer  $j$ satisfying $2\leqslant j\leqslant k+1$.
Finally,
for every   $k$-subset $S$, it follows from \eqref{4lambda} that
$$0=\mathnormal{\Gamma}\big(tf_S( t ),f_{M\setminus S}( t )\big)=\lambda_{01}\mathnormal{\Gamma}_{01} \big(tf_S( t ),f_{M\setminus S}( t )\big)=-\lambda_{01}f_S(0)f_{M\setminus S}(0)=(2k+1)!\mathnormal{\Gamma}(  1,  t ),$$
which means that  $\mathnormal{\Gamma} (  1,  t )=0$. Therefore,  all coefficients in \eqref{2lambda} vanish, completing the proof.
\end{proof}

Combining \eqref{m-me} with Theorem \ref{o-g-l-b}, we obtain
$
m(\mathbbmsl{O}_k,2)
\geqslant
(k^2+3k+2)/4$.
To complete the proof of Theorem \ref{OO}, it remains to
establish the corresponding upper bound, as stated in
the following theorem.

\begin{theorem}\label{OGOG}
For every  positive    integer $k$,
\[ m(\mathbbmsl{O}_k,2)\leqslant\left\lceil\mfrac{k^2+3k+2}{4}\right\rceil.\]
\end{theorem}

\begin{proof}
Fix a positive integer $k$  and let the vertices of $\mathbbmsl{O}_k$ be identified with the $k$-subsets of a set $M$ of size $2k+1$.
We recall that two vertices of $\mathbbmsl{O}_k$ are
adjacent if and only if they are disjoint.
The theorem then follows if we can construct a percolating set of size at most $\lceil(k^2+3k+2)/4\rceil$ under the $2$-neighbor bootstrap percolation process on $\mathbbmsl{O}_k$.
When $k=1$, the  graph  \(\mathbbmsl{O}_1\) is  the complete graph on  three     vertices and therefore,  in the \(2\)-neighbor bootstrap percolation process on  \(\mathbbmsl{O}_1\), any two initially active vertices activate the third vertex. Henceforth, we assume that \(k\geqslant2\) and set
$p=\lfloor k/2\rfloor$ and  $q=\lceil k/2\rceil$.
Partition \(M\) into   two  sets
$$X = \{x_0, x_1,    \ldots, x_{2p-1}\}  \quad  \text{ and } \quad    Y  = \{y_0, y_1,    \ldots, y_{2q}\}.$$
Interpret subscripts on elements of $Y$ modulo $2q+1$ throughout.
For every  integer \(i\in\{ 0, 1,\ldots,p\}\),       define
\[
C_i=\{x_0, x_1, \ldots, x_{2i-1}\}
\quad
\text{ and }
\quad
D_i=\{x_{2i},   x_{2i+2},  \ldots,x_{2p-2}\}.
\]
Note that $C_0 = \varnothing$ and $D_p = \varnothing$ by convention.
Let   $$\mathcal{Z}_i=\big\{Z\cup D_i  \, \big| \,  Z \text{ is an  $i$-subset of } C_i\big\}.$$
Observe that  $\mathcal{Z}_i$ forms a family of $p$-subsets of $X$ for all $i$.
For every  integer $i \in \{0, 1, \ldots, p\}$, let
\[
R_i=\{x_1, x_3, \ldots,x_{2i-1}\}\cup D_i.
\]
Notice that   $R_i\in  \mathcal{Z}_i$ and consequently    $R_i$ is a $p$-subset of $X$ for all $i$.
For each     integer  \(i\) and each    subset \(R\subseteq X\),  define
$$
L_i(R)=R\cup \{y_i,  \ldots,y_{i+k-|R|-1}\}$$ and
$$U_i(R)=(X\setminus R)\cup \{y_{i+k-|R|+1},  \dots,  y_{i+2q}\}.
$$
Both are \(k\)-subsets of \(M\).
It is straightforward to verify that either equality \(L_i(R)=L_{i'}(R')\) or \(U_i(R)=U_{i'}(R')\) forces  \(R=R'\). Also, if either \(L_i(R)=L_{i'}(R)\) with \(R\neq X\)  or \(U_i(R)=U_{i'}(R)\) with \(R\neq\varnothing\), then \(i\equiv i'\pmod{2q+1}\).

Let \(\mathcal{F}\) be the union of the following five sets
\begin{align*}
\mathcal{F}_1=  {} &  \big\{L_q(R_0), L_q(R_0\setminus\{x_0\})\big\},\\[6pt]
\mathcal{F}_2=  {} &  \Big\{L_q\big((R_0\setminus\{x_0\})\cup\{x_{2i+1}\}\big) \, \Big| \, 0\leqslant  i\leqslant  p-1\Big\},\\[6pt]
\mathcal{F}_3={}&\Big\{U_q\big((R_0\setminus\{x_{2i}\})\cup\{x_{4j+1},x_{4j+3}\}\big)\, \Big| \, 1\leqslant  i\leqslant  p-1 \text{ and } 0\leqslant  j\leqslant  \left\lfloor \mfrac{p}{2}\right\rfloor -1 \Big\},\\[6pt]
\mathcal{F}_4={}&
\begin{cases}\\[-13pt]
\displaystyle{\Big\{U_q\big((R_0\setminus\{x_{2i}\})\cup\{x_{2p-1}\}\big) \, \Big| \,  1\leqslant  i\leqslant p-1\Big\}}
       &\text{ if } p \text{ is odd}, \\[18pt]
\displaystyle{\Big\{L_q\big((R_0\setminus\{x_{4i},x_{4i+2}\})\cup\{x_{2p-1}\}\big)\, \Big| \,   0 \leqslant i \leqslant    \mfrac{p}{2} -1\Big\}}
       &\text{ otherwise},\\[9pt]
\end{cases}\\[6pt]
\mathcal{F}_5={}&\big\{L_i(R_j)  \, \big| \,  1\leqslant  i\leqslant  q-1 \text{ and }  0\leqslant j\leqslant  i\big\}.
\end{align*}
The cardinalities of these sets are given by $|\mathcal{F}_1|=2$,   $ |\mathcal{F}_2|=p$, $ |\mathcal{F}_3|=(p-1)\lfloor p/2\rfloor$,
$$ |\mathcal{F}_4|=\begin{cases}
\displaystyle{p-1}
       &\text{ if } p \text{ is odd}, \\[12pt]
\displaystyle{\mfrac{p}{2}}
       &\text{ otherwise},
\end{cases}   \quad  \text{ and } \quad   |\mathcal{F}_5|=\sum_{i=1}^{q-1}(i+1)=\mfrac{q^2+q-2}{2}.$$
Thus, substituting
\[
(p,q)=
\begin{cases}\\[-15pt]
\displaystyle{\left(\mfrac{k}{2},\mfrac{k}{2}\right)} & \text{ if }k\text{ is even},\\[18pt]
\displaystyle{\left(\mfrac{k-1}{2},\mfrac{k+1}{2}\right)} & \text{ otherwise},\\[7pt]
\end{cases}
\]
and simplifying yields
\[
|\mathcal{F}|\leqslant |\mathcal{F}_1|+\cdots+|\mathcal{F}_5|=
\begin{cases}
\displaystyle{\mfrac{k^2+3k+2}{4}} & \text{ if }p\text{ is odd},\\[9pt]
\displaystyle{\mfrac{k^2+3k+4}{4}} & \text{ otherwise}.
\end{cases}
\]
This implies  that \(|\mathcal{F}|\leqslant \lceil(k^2+3k+2)/4\rceil.\)
Hence,  to complete the proof, it suffices to show that the initial activation of the vertices in \(\mathcal{F}\) eventually activates all the remaining vertices of  \(\mathbbmsl{O}_k\).
To streamline the proof and clarify the arguments, we state and prove several lemmas. We begin with the following lemma, which will be used repeatedly throughout the proof, though not always explicitly cited.

\begin{lemma}\label{local}
Let  \(i\) be an integer  and   \(R\subseteq X\). Then,  the vertex \(L_i(R)\) is adjacent to the vertices  \(U_i(R)\), \(U_{i-1}(R)\), and \(U_i(R\cup\{x\})\) for each element  \(x\in X\setminus R\).
\end{lemma}

\begin{proof}[Proof of Lemma \ref{local}]
\renewcommand{\qedsymbol}{\rotatebox[origin=c]{45}{$\Box$}}
The assertion follows immediately from
\begin{equation*}M\setminus L_i(R)=U_i(R)\cup\{y_{i+k-|R|}\}=U_{i-1}(R)\cup\{y_{i+2q}\}=U_i\big(R\cup\{x\}\big)\cup\{x\}.\qedhere\end{equation*}
\end{proof}

\begin{lemma}\label{replacement}
Let \(W\) be a \(k\)-subset of \(M\) and let
\(A\subseteq W\)  and  \(B\subseteq M\setminus W\).   In the \(2\)-neighbor bootstrap percolation process on  \(\mathbbmsl{O}_k\),
if   \((W\setminus\{a\})\cup\{b\}\) is active  for all   elements  \(a\in A\) and \(b\in B\),  then
$(W\setminus S)\cup T$  eventually becomes active  for all  subsets
$S\subseteq A$ and $T\subseteq B$ satisfying     $|S|=|T|\geqslant1$.
\end{lemma}

\begin{proof}[Proof of Lemma \ref{replacement}]
Let $\ell = \vert{}S\vert{} = \vert{}T\vert{}$. We proceed by induction on $\ell$. The case \(\ell=1\) is trivial.
Suppose  that  $\ell\geqslant2$ and
let
$S\subseteq A$ and $T\subseteq B$ be   $\ell$-subsets.
Choose distinct elements
\(s_1,s_2\in S\) and   \(t_1,t_2\in T\).
By the induction hypothesis, two  vertices $(W\setminus (S\setminus\{s_1\}))\cup (T\setminus\{t_1\})$ and $(W\setminus (S\setminus\{s_1\}))\cup (T\setminus\{t_2\})$ are   active. Consequently, the activation rule implies that
the  vertex  $(M\setminus  (W\cup  T))\cup (S\setminus\{s_1\})$ eventually becomes active.
Similarly,
the  vertex  $(M\setminus  (W\cup  T))\cup (S\setminus\{s_2\})$ eventually becomes active. Since both of these vertices are now active,  the activation rule implies that  the vertex $(W \setminus S) \cup T$   also   becomes active.
\renewcommand{\qedsymbol}{\rotatebox[origin=c]{45}{$\Box$}}
\end{proof}

\begin{lemma}\label{endpoints}
For every  subset \(R\subseteq X\) with $|R|\in\{p-1, p\}$, the vertex    \(L_q(R)\)  eventually    becomes active  from  \(\mathcal{F}\) under  the \(2\)-neighbor bootstrap percolation process on \(\mathbbmsl{O}_k\).
\end{lemma}

\begin{proof}[Proof of Lemma \ref{endpoints}]
We first establish  the claim that,  for any  integers \(i,j\in\{0,1,\ldots,p-1\}\), the vertex
$L_q ((R_0\setminus\{x_{2i}\})\cup\{x_{2j+1}\} )$
eventually becomes active.
Since \(L_q((R_0\setminus\{x_0\})\cup\{x_{2j+1}\})\in\mathcal{F}_2\), we may  assume    that  \(i\geqslant 1\).
The vertex
$U_q(R_0\cup\{x_{2j+1}\})$ eventually becomes active   because, by Lemma \ref{local},  it has two distinct active neighbors
$L_q(R_0)\in\mathcal{F}_1$ and  $ L_q((R_0\setminus\{x_0\})\cup\{x_{2j+1}\})\in\mathcal{F}_2$. If   $p$ is odd and   $j= p-1$,
then
$L_q ((R_0\setminus\{x_{2i}\})\cup\{x_{2p-1}\} )$ eventually  becomes active   because, by Lemma \ref{local},  it is adjacent to both of the distinct active  vertices
$U_q((R_0\setminus\{x_{2i}\})\cup\{x_{2p-1}\})\in\mathcal{F}_4$ and $U_q(R_0\cup\{x_{2p-1}\})$.
We may therefore assume that either  $p$  is even or
\(j\in\{0,1,\ldots,p-2\}\). Then,  there exists an integer    ${j'}$
such that   $ 0\leqslant  {j'}\leqslant   \lfloor  p/2 \rfloor -1$ and  $j\in\{2{j'}, 2{j'}+1\}$.
Consequently,
$L_q((R_0\setminus\{x_{2i}\})\cup\{x_{2j+1}\})$ eventually  becomes active   because, by Lemma \ref{local},  it is adjacent to both of the distinct active  vertices
$U_q((R_0\setminus\{x_{2i}\})\cup\{x_{4{j'}+1},x_{4{j'}+3}\})\in\mathcal{F}_3$ and $U_q(R_0\cup\{x_{2j+1}\})$.   Thus, the claim is proved.

We next   establish  the claim that,   for any  integer  \(i \in\{0,1,\ldots,p-1\}\), the vertex
$L_q (R_0\setminus\{x_{2i}\} )$
eventually becomes active. For $i=0$, this vertex belongs to $\mathcal{F}_1$, so assume $i\geqslant1$. The vertex
$U_q(R_0)$ eventually  becomes active   because, by Lemma \ref{local},  it has two distinct active neighbors   $L_q(R_0), L_q(R_0\setminus\{x_0\})\in \mathcal{F}_1$. Since
$L_q (R_0\setminus\{x_{2i}\} )$
is adjacent to both of the distinct   vertices   $U_q(R_0)$ and $U_q((R_0\setminus\{x_{2i}\})\cup\{x_{2p-1}\})$, it suffices to show that the latter vertex eventually becomes active. If $p$ is odd, then  $U_q((R_0\setminus\{x_{2i}\})\cup\{x_{2p-1}\})\in \mathcal{F}_4$, proving  the claim.  We may therefore   assume that     $p$ is even.
Then,  there exists an integer    ${i'}$ such that   $ 0\leqslant  {i'}\leqslant     p/2   -1$ and  $i\in\{2{i'}, 2{i'}+1\}$.
Consequently,
$U_q((R_0\setminus\{x_{2i}\})\cup\{x_{2p-1}\})$ eventually  becomes active because, by Lemma \ref{local},
it has two distinct  active  neighbors
$L_q((R_0\setminus\{x_{4i^{\prime}},x_{4i^{\prime}+2}\})\cup\{x_{2p-1}\})\in\mathcal{F}_4$ and $L_q((R_0\setminus\{x_{2i}\})\cup\{x_{2p-1}\})$.  Hence,  the claim  is established.

Now, fix a   subset  \(R\subseteq X\) with $|R|\in\{p-1, p\}$. We prove  that  the vertex  $ L_q(R)   $ eventually becomes active.  If $R_0\subseteq  R$, then $R=R_0$ which implies that    $ L_q(R)\in\mathcal{F}_1   $, we are done.  Otherwise, we apply  Lemma   \ref{replacement} with   \(W=L_q(R_0)\),
\(A=R_0\),   \(B= (X\setminus R_0)\cup\{y_{2q}\}\),
$S=R_0\setminus  R$, and
\[
T=
\begin{cases}
(X\setminus  R_0) \cap R & \text{ if }|R|=p,\\[9pt]
((X\setminus  R_0) \cap R)\cup\{y_{2q}\} & \text{ otherwise}.
\end{cases}
\]
Since
both vertices
$(W\setminus \{x_{2i}\}) \cup \{x_{2j+1}\}=L_q ((R_0\setminus\{x_{2i}\})\cup\{x_{2j+1}\} ) $ and
$(W\setminus \{x_{2i}\}) \cup \{y_{2q}\}=L_q (R_0\setminus\{x_{2i}\} ) $
are  active     for all   integers \(i,j\in\{0,1,\ldots,p-1\}\)   and since   $|S|=|T|\geqslant1$,
Lemma   \ref{replacement}  guarantees that   $ L_q(R)=(W\setminus S) \cup T  $ eventually becomes active.
\renewcommand{\qedsymbol}{\rotatebox[origin=c]{45}{$\Box$}}
\end{proof}

\begin{lemma}\label{endpoints-prime}
For every \(p\)-subset \(R\subseteq X\), the vertices \(L_0(R)\),   \(U_q(R)\),    and \(U_{2q}(R)\)  eventually   become active  from  \(\mathcal{F}\) under  the \(2\)-neighbor bootstrap percolation process on \(\mathbbmsl{O}_k\).
\end{lemma}

\begin{proof}[Proof of Lemma \ref{endpoints-prime}]
Fix a \(p\)-subset \(R\subseteq X\). It follows  from   Lemma \ref{endpoints} that
both  vertices \(U_{2q}(R)=U_{-1}(R)=L_q(X\setminus R)\) and
\(U_0(R\cup\{x\})=L_q(X\setminus(R\cup\{x\}))\) are
active, where \(x\in X\setminus R\). By Lemma \ref{local}, the
vertex \(L_0(R)\) has two distinct active neighbors, namely
\(U_{-1}(R)\) and \(U_0(R\cup\{x\})\), and hence also becomes
active. Finally, since \(U_q(R)=L_0(X\setminus R)\), we
conclude that \(U_q(R)\) eventually becomes active as well.
\renewcommand{\qedsymbol}{\rotatebox[origin=c]{45}{$\Box$}}
\end{proof}

\begin{lemma}\label{adjoin}
Let \(P, Q\subseteq X\)   be disjoint subsets, let    \(x\in X\setminus(P\cup Q)\),
and    let  \(t\) be an integer satisfying \(1\leqslant  t\leqslant |Q|-1\).
Suppose  that
\begin{itemize}[noitemsep,topsep=0pt]
\item[{\rm (i)}]  \(U_i(P\cup R)\) is active  for every  integer \(i\in\{0,1,\ldots,t\}\) and   every $t$-subset $R\subseteq Q$;
\item[{\rm (ii)}] \(L_0(P\cup R \cup \{x\})\) and \(L_{t+1}(P\cup R \cup \{x\})\) are active    for every $(t-1)$-subset $R\subseteq Q$.
\end{itemize}
Then,  \(L_i(P\cup R \cup \{x\})\) eventually becomes active    for each   integer \(i\in\{1,\ldots,t\}\) and   each  $(t-1)$-subset $R\subseteq Q$.
\end{lemma}

\begin{proof}[Proof of Lemma \ref{adjoin}]
We proceed by induction on $t$ and   begin with the case  $t = 1$.
To establish that  \(L_1(P\cup  \{x\})\) eventually becomes active,   Lemma \ref{local} shows it suffices to verify that $U_0(P \cup \{x\})$  and $U_1(P \cup \{x\})$  become    active.
Since   $\vert{}Q\vert{} \geqslant 2$, we can  choose   two distinct elements $z_1, z_2 \in Q$.  For each integer  $i \in \{0, 1\}$,   condition  (i) ensures the activation of the   distinct vertices  $U_i(P \cup \{z_1\})$ and $U_i(P \cup \{z_2\})$  which in turn activate   $L_i(P)$.
Since $L_0(P \cup \{x\})$ is active  via (ii), the vertex $U_0(P \cup \{x\})$ has two distinct active  neighbors  $L_0(P)$ and $L_0(P \cup \{x\})$ and thereby becomes active.
Analogously,  since $L_2(P \cup \{x\})$ is active  via (ii), the vertex $U_1(P \cup \{x\})$ possesses two distinct active    neighbors $L_1(P)$ and $L_2(P \cup \{x\})$ and   thus becomes active.  This completes  the proof for the  case $t = 1$.

Let $t \geqslant 2$ and assume that the assertion holds for $t-1$. We aim to establish it for $t$. To this end, suppose that conditions (i) and (ii) hold for $t$. In order to apply the induction hypothesis, we must demonstrate  that these assumptions imply that conditions (i) and (ii) also hold for $t-1$.
To do  this, we first establish the  claim that $L_i(P \cup R)$ is active  for all integers $i \in \{0, 1, \ldots, t\}$ and all  $(t-1)$-subsets $R \subseteq Q$. Fix a $(t-1)$-subset $R \subseteq Q$. Since $\vert{}Q\vert{} \geqslant t+1$, we can choose two distinct elements   $z_1, z_2 \in Q \setminus R$. By condition (i), the  vertices $U_i(P \cup R \cup \{z_1\})$ and $U_i(P \cup R \cup \{z_2\})$ are distinct and active  for each  integer  $i \in \{0, 1, \ldots, t\}$. These two vertices together activate  $L_i(P \cup R)$, thereby establishing the claim.
Next, to show that condition (i) holds for $t-1$, observe that for each integer $i \in \{0, 1, \ldots, t-1\}$ and each $(t-1)$-subset $R \subseteq Q$, the vertex $U_i(P \cup R)$ has   two distinct active   neighbors   $L_i(P \cup R)$ and $L_{i+1}(P \cup R)$ and  thus  becomes active. This confirms that condition (i) is valid for $t-1$.
Finally, we verify condition (ii)     for $t-1$. Fix an integer  $i\in\{0, t\}$ and a $(t-2)$-subset $R \subseteq Q$. Let  $s=0$ if $i=0$ and $s=t+1$ if $i=t$.  As $\vert{}Q\vert{} \geqslant t+1$, there exist two distinct elements $z_1, z_2 \in Q \setminus R$. For each integer $\ell\in\{1, 2\}$, the  active  vertices     $L_i(P  \cup R\cup \{z_\ell\})$ and $L_s(P  \cup R\cup \{x, z_\ell\})$   are adjacent to   \(U_i(P\cup R \cup \{x, z_\ell\})\). It follows  from   Lemma \ref{local} that   $L_i(P  \cup R\cup \{x\})$ has two distinct active  neighbors of the form  \(U_i(P\cup R \cup \{x, z_\ell\})  \) for $\ell\in\{1, 2\}$  and hence becomes active. Therefore, condition (ii) is valid for $t-1$.

With   conditions (i) and (ii)   for $t-1$ established, we may now apply the induction hypothesis. Fix    a   $(t-1)$-subset $R\subseteq Q$. As $|R|\geqslant 1$, we may choose an element $z\in R$. Using  the induction hypothesis and condition (ii) for $t-1$, the vertex   \(L_i(P\cup (R\setminus\{z\}) \cup \{x\})\) is active   for   all   integers    \(i\in\{0, 1,\ldots,t\}\). Recall from the previous paragraph that  $L_i(P \cup R)$ is also active  for all  integers   $i \in \{0, 1, \ldots, t\}$.
It follows  from   Lemma \ref{local} that   $U_i(P  \cup R\cup \{x\})$ has two distinct  active  neighbors  $L_i(P \cup R)$ and   \(L_i(P\cup (R\setminus\{z\}) \cup \{x\})\)   and hence becomes active. Therefore, for every     integer    \(i\in\{1,\ldots,t\}\), the vertex   $L_i(P  \cup R\cup \{x\})$ has two distinct  active  neighbors  \(U_i(P\cup R \cup \{x\})  \) and \(U_{i-1}(P\cup R \cup \{x\})  \), so it eventually becomes active as well.
\renewcommand{\qedsymbol}{\rotatebox[origin=c]{45}{$\Box$}}
\end{proof}

\begin{lemma}\label{5prime}
Let $t$ be an integer with \(2\leqslant t\leqslant  p\),  let \(P, Q\subseteq X\) be   disjoint subsets   with
\(|P|=p-t\) and  \(|Q|=2t\), and let  \(x\in Q\). Suppose that
\begin{itemize}[noitemsep,topsep=0pt]
\item[{\rm (i)}]  \(L_i(P\cup R)\) is active   for  any    integer \(i\in\{0,1,\ldots,t\}\) and  any    $t$-subset $R\subseteq Q$ containing  $x$;
\item[{\rm (ii)}]    \(U_t(P\cup R )\) and \(U_{2q}(P\cup R)\)   are  active     for every  $t$-subset $R\subseteq Q\setminus\{x\}$.
\end{itemize}
Then,  \(U_i(P\cup R )\) eventually becomes active    for each   integer \(i\in\{0,1,\ldots,t-1\}\) and   each  $t$-subset $R\subseteq Q\setminus\{x\}$.
\end{lemma}

\begin{proof}[Proof of Lemma \ref{5prime}]
We state and prove three auxiliary claims.
We first establish the claim that \(U_i(P\cup R)\) eventually becomes active for all     integers $i, j$ satisfying    \(0\leqslant j\leqslant i\leqslant t\) and  for every   \((t+j+1)\)-subset \(R\subseteq Q\) containing \(x\). We prove this claim by induction on \(j\).
Fix a \((t+j+1)\)-subset \(R\subseteq Q\) containing \(x\). Since    \(t\geqslant 2\), there exist two distinct elements \(a_1,a_2\in R\setminus\{x\}\).
Consider   the   case \(j=0\).   By    (i), the vertex  \(L_i(P\cup(R\setminus\{a_\ell\}))\) is active  for any   integers \(i\in\{0,1,\ldots,t\}\) and   \(\ell\in\{1,2\}\). Consequently, \(U_i(P\cup R)\) has two distinct active  neighbors, namely \(L_i((P\cup R)\setminus\{a_1\})\) and \(L_i((P\cup R)\setminus\{a_2\})\).  It therefore eventually becomes active for every integer      \(i\) with \(0 \leqslant i\leqslant t\). This establishes the claim for \(j=0\).
Now, suppose that \(j\geqslant 1\) and that the claim holds for \(j-1\). By the induction hypothesis, both vertices  \(U_i(P\cup(R\setminus\{a_\ell\}))\) and \(U_{i-1}(P\cup(R\setminus\{a_\ell\}))\) are eventually active   for   any   integers \(i\in\{j,\ldots,t\}\) and      \(\ell\in\{1,2\}\).
Accordingly, for each integer    \(\ell\in\{1,2\}\), the vertex  \(L_i(P\cup(R\setminus\{a_\ell\}))\) has two distinct active  neighbors
and  thus  eventually becomes active whenever   \(i\in\{j,\ldots,t\}\). It follows that  \(U_i(P\cup R)\) has two distinct active  neighbors, namely  \(L_i((P\cup R)\setminus\{a_1\})\) and \(L_i((P\cup R)\setminus\{a_2\})\). Hence,     \(U_i(P\cup R)\)     eventually becomes active for each  integer \(i\in\{j,\ldots,t\}\).  This completes the induction and the proof of the claim.

We next  establish the claim that \(L_j(P\cup R)\) eventually becomes active for every    integer    \(j\in\{0, 1, \ldots,  t-1\}\) and every    \((t+j)\)-subset \(R\subseteq Q\setminus\{x\}\). We prove this claim by induction on \(j\).
Fix  a  \((t+j)\)-subset \(R\subseteq Q\setminus\{x\}\).
Consider   the   case \(j=0\).  By (ii) and the preceding paragraph,   \(U_{-1}(P\cup R)\) and  \(U_0(P\cup R\cup\{x\})\)
are both eventually  active.
Using  Lemma \ref{local},    \(L_0(P\cup R)\) has two distinct active neighbors and hence   eventually becomes active. This establishes the claim for \(j=0\).
Now, suppose that \(j\geqslant 1\) and that the claim holds for \(j-1\). From  \(t+j\geqslant 3\), we may select  two distinct elements \(b_1,b_2\in R\).
By the induction hypothesis, \(L_{j-1}(P\cup(R\setminus\{b_\ell\}))\)  is eventually active  for each integer \(\ell\in\{1,2\}\), which implies that \(U_{j-1}(P\cup R)\) eventually becomes active. On the other hand, \(U_j(P\cup R\cup\{x\})\) is eventually active by the preceding paragraph. It follows that \(L_j(P\cup R)\) has two distinct active neighbors, namely \(U_{j-1}(P\cup R)\) and \(U_j(P\cup R\cup\{x\})\), and thus eventually becomes active. This completes the induction and the proof of the claim.

We finally establish that \(L_t(P\cup R)\) eventually becomes active for every subset \(R\subseteq Q\setminus\{x\}\) with \(|R|\geqslant t\). To do  this, fix such a subset \(R\) and write \(|R|=t+j\) for some integer  \(j\in\{0,1,\ldots,t-1\}\). We proceed by induction on \(j\).
Consider   the   case \(j=0\).  By (ii) and the first  paragraph of the proof,   \(U_{t}(P\cup R)\) and  \(U_{t}(P\cup R\cup\{x\})\)
are both eventually  active.
Using  Lemma \ref{local},    \(L_t(P\cup R)\) has two distinct active neighbors and hence   eventually becomes active. This establishes the claim for \(j=0\).
Now, suppose that \(j\geqslant 1\) and that the claim holds for \(j-1\). From \(t+j\geqslant 3\), we may choose two distinct elements \(c_1,c_2\in R\).
By the induction hypothesis, \(L_{t}(P\cup(R\setminus\{c_\ell\}))\)  is eventually active  for each integer \(\ell\in\{1,2\}\), which implies that \(U_{t}(P\cup R)\) eventually becomes active. On the other hand, \(U_t(P\cup R\cup\{x\})\) is eventually active by the first  paragraph of the proof. It follows that \(L_t(P\cup R)\) has two distinct active neighbors, namely \(U_{t}(P\cup R)\) and \(U_t(P\cup R\cup\{x\})\), and thus eventually becomes active. This completes the induction and the proof of the claim.

We now   argue by contradiction. Suppose that there exist  an    integer
\(i\in\{0,1,\ldots,t-1\}\) and        a \(t\)-subset
\(R\subseteq Q\setminus\{x\}\) such that
\(U_{i}(P\cup R)\) never becomes active.
Since   \(U_{i}(P\cup R)\) is  adjacent to  both  vertices \(L_{i}(P\cup R )\) and \(L_{i+1}(P\cup R )\),  at least one
of these two vertices never becomes active. Consequently,    there exist
an integer \(i\in\{0,1,\ldots,t\}\) and a \(t\)-subset
\(R \subseteq Q\setminus\{x\}\) such that
\(L_{i}(P\cup R )\) never becomes active.
However,  by    (ii),   the first  paragraph of the proof,  and Lemma \ref{local},  the vertices
\(U_t(P\cup R )\) and \(U_t(P\cup R \cup\{x\})\)   are  two distinct
active neighbors of
\(L_t(P\cup R )\).  It follows that
$L_t(P\cup R)$ eventually becomes active  and hence $i\leqslant t-1$.
Let \(j\) be the largest integer such that there exist an integer  \(i\in\{0,1,\ldots,t-1\}\) and a
\((t+j)\)-subset \(R\subseteq Q\setminus\{x\}\) for which  \(0\leqslant j\leqslant  i \) and
\(L_{i}(P\cup R)\) never becomes active.
Since   \(L_{j}(P\cup  R )\) eventually becomes active by the  second paragraph of the proof,   we must have  $i\geqslant j+1$.
It follows from   $j\leqslant i-1 \leqslant t-2$ that   there is   an element $d\in Q \setminus(R  \cup \{x\})$.
By the maximality of $j$ and by the preceding paragraph, both
vertices  \(L_{i}(P\cup (R \cup \{ d \}))\) and \(L_{i+1}(P\cup ( R \cup \{ d \} ) )\) eventually become active. However, both of these vertices
are adjacent to  \(U_{i}(P\cup   R \cup \{ d \}) \) and  thus  \(U_{i}(P\cup R\cup\{d\})\) eventually becomes active as well.
On the other hand,       \(U_{i}(P\cup R\cup\{x\})\) eventually becomes active by  the  first paragraph of the proof.
It follows that \(L_i(P\cup R)\) has two distinct active neighbors, namely \(U_{i}(P\cup R\cup\{d\})\)  and  \(U_{i}(P\cup R\cup\{x\})\) and hence  eventually becomes active,   which is a contradiction.
\renewcommand{\qedsymbol}{\rotatebox[origin=c]{45}{$\Box$}}
\end{proof}

\begin{lemma}\label{ChDh}
For   any  integers $i  \in \{0, 1, \ldots, q-1\}$
and $j  \in \{ 1, \ldots, p\}$ and for any  subset   \(R\in  \mathcal{Z}_j\), the   vertices   \(L_i(R)\) and \(U_i(R)\)    eventually become  active     from  \(\mathcal{F}\) under  the \(2\)-neighbor bootstrap percolation process on \(\mathbbmsl{O}_k\).
\end{lemma}

\begin{proof}[Proof of Lemma \ref{ChDh}]
We prove the assertion by induction on $j$.
We begin with the   case  \(j=1\).
Note that \(\mathcal{Z}_1=\{R_0, R_1\}\). Fix an  integer $\ell\in\{0, 1\}$.
Since $\{L_i(R_\ell) \mid 1 \leqslant i \leqslant q-1\} \subseteq \mathcal{F}_5$ and Lemmas \ref{endpoints} and \ref{endpoints-prime} guarantee that $L_0(R_\ell)$ and $L_q(R_\ell)$  eventually become active, it follows that $L_i(R_\ell)$ is active  for every integer  $i \in \{0, 1, \ldots, q\}$. Furthermore, for each \(i\in\{0,1,\ldots,q-1\}\), the vertex \(U_i(R_\ell)\) has two distinct active  neighbors, namely \(L_i(R_\ell)\) and \(L_{i+1}(R_\ell)\), and hence   becomes active as well.
This completes the proof of the   case \(j=1\).

Let $j  \geqslant 2$ and assume that the assertion holds for $j-1$.
We establish the assertion for \(j\) in two stages. In the first stage, we prove that \(L_i(R)\) and \(U_i(R)\) eventually become active for every integer \(i\in\{j,\ldots,q-1\}\) and every subset \(R\in\mathcal{Z}_j\).
We begin with the special case $R=R_j$.
Since $\{L_i(R_j)  \,  | \,  j\leqslant  i\leqslant  q-1 \}\subseteq \mathcal{F}_5$  and Lemma \ref{endpoints} ensures that $L_q(R_j)$     eventually becomes active, it follows that $L_i(R_j)$ is active   for each  integer \(i\in\{j,\ldots,q\}\).  In addition, for each \(i\in\{j,\ldots,q-1\}\), the vertex \(U_i(R_j)\) has two distinct active  neighbors, namely \(L_i(R_j)\) and \(L_{i+1}(R_j)\), and hence   becomes active as well.

Fix  an integer      \(i\in\{j,\ldots,q-1\}\).    To apply  Lemma   \ref{replacement}, let  \(W=L_i(R_{j-1})\),
\(A=R_{j-1}\setminus D_j\),   and $B= C_j\setminus R_{j-1}$.
Note  that \(A\subseteq W\)  and  \(B\subseteq M\setminus W\).
We establish the claim that the vertex  $(W \setminus \{a\}) \cup \{b\} = L_i((R_{j-1} \setminus \{a\}) \cup \{b\})$ eventually becomes active  for any
elements   $a \in A$ and $b \in B$.
Since our  proof of this claim relies on two auxiliary vertices, we first demonstrate   that \(L_i(R_{j-1}\setminus\{a\})\) and \(U_i(R_{j-1}\cup\{b\})\) eventually become active.
Set \(a'=x_{2j-1}\) if \(a=x_{2j-2}\)  and \(a'=x_0\) otherwise. Then,
\((R_{j-1}\setminus\{a\})\cup\{a'\}\) either coincides with \(R_j\) or belongs to \(\mathcal{Z}_{j-1}\). In either case, we already know that \(U_i((R_{j-1}\setminus\{a\})\cup\{a'\})\) eventually becomes active. Therefore, \(L_i(R_{j-1}\setminus\{a\})\) has two distinct active  neighbors, namely \(U_i(R_{j-1})\) and \(U_i((R_{j-1}\setminus\{a\})\cup\{a'\})\), and hence becomes active.
Set \(b'=x_{2j-2}\) if \(b=x_{2j-1}\)  and \(b'=x_1\) otherwise. Then,
\((R_{j-1}\cup\{b\})\setminus\{b'\}\) either coincides with \(R_j\) or belongs to \(\mathcal{Z}_{j-1}\). In either case, we already know that \(L_i((R_{j-1}\cup\{b\})\setminus\{b'\})\) eventually becomes active. Therefore, $U_i(R_{j-1}\cup\{b\})$ has two distinct active  neighbors, namely \(L_i(R_{j-1})\) and \(L_i((R_{j-1}\cup\{b\})\setminus\{b'\})\), and hence   becomes active.

We are now ready to finish the proof of the preceding claim. For the sake of simplicity,    we  set \(R'=(R_{j-1}\setminus\{a\})\cup\{b\}\).
The vertex \(U_{q-1}(R')\) has two distinct active  neighbors, namely    \(L_q(R')\) and \(L_{q-1}(R_{j-1}\setminus\{a\})\), and thus    becomes active.
Then,    \(L_{q-1}(R')\) has two distinct active  neighbors, namely   \(U_{q-1}(R')\) and \(U_{q-1}(R_{j-1}\cup\{b\})\), and hence also    becomes active.
Iterating this argument,  we conclude that   the vertices \(L_{q-1}(R'), L_{q-2}(R'),\ldots,L_j(R')\)  all eventually become active  in this order.
This completes the proof of the claim, establishing     that  the vertex  $(W \setminus \{a\}) \cup \{b\} = L_i(R')$ eventually becomes active  for any
elements   $a \in A$ and $b \in B$.

We are now in a position to apply   Lemma \ref{replacement}. To this end, fix  an arbitrary    subset \(R\in\mathcal{Z}_j\) and   let \(S=A\setminus R\) and \(T=B\cap R\).
If \(S=\varnothing\), then \(R=R_{j-1}\) and  hence    the induction hypothesis yields that  both vertices \(L_i(R)\) and \(U_i(R)\) eventually become active. Thus,  we   assume that \(S\neq\varnothing\).
Since $A$, $B$, and $ C_j\cap R$  are  all   $j$-subsets and  since  $\{A, B\}$ forms a partition of  $C_j$, we obtain
$$1\leqslant |S|=|A|-|A\cap R| =j-\big(|C_j\cap R|-|B\cap R|\big)=|T|.$$ Moreover, we have
$$(R_{j-1}\setminus S)\cup T=\big((A\cup D_j)\setminus (A\setminus R)\big)\cup(B\cap R)=(A\cap R)\cup  (B\cap R) \cup D_j =(C_j\cap R)\cup D_j=R.$$
Now,       Lemma \ref{replacement} guarantees   that  the vertex
$(W\setminus S)\cup T=L_i((R_{j-1}\setminus S)\cup T)=L_i(R)$ eventually becomes active. Furthermore, since   $L_q(R)$ is  active   by Lemma \ref{endpoints}  and since
\(U_i(R)\) has two distinct active  neighbors \(L_i(R)\) and \(L_{i+1}(R)\),  we find  that  \(U_i(R)\)  eventually becomes active.

In the second  stage, we prove that \(L_i(R)\) and \(U_i(R)\) eventually become active for every integer \(i\in\{0, 1,\ldots,j-1\}\) and every subset \(R\in\mathcal{Z}_j\).
To apply Lemma \ref{adjoin}, we consider the substitutions $P = D_{j-1}$, $Q = C_{j-1}$, $x = x_{2j-1}$, and $t = j-1$.
The induction hypothesis guarantees that condition (i) of Lemma \ref{adjoin}  is satisfied, since \(P\cup R\in\mathcal{Z}_{j-1}\) for every \(t\)-subset \(R\subseteq Q\). 
Further, Lemmas \ref{endpoints} and \ref{endpoints-prime}, together with the first stage, ensure that
condition (ii) of Lemma \ref{adjoin} is satisfied, since \(P\cup R\cup\{x\}\in\mathcal{Z}_j\) for every \((t-1)\)-subset \(R\subseteq Q\).
It now follows from Lemma \ref{adjoin} that $L_i(R)$ eventually becomes  active   for each integer $i \in \{1, \ldots, j-1\}$ and each subset $R \in \mathcal{Z}_j$ containing both elements  $x_{2j-2}$ and $x_{2j-1}$. On the  other hand,  if    a subset    $R \in \mathcal{Z}_j$  contains     $x_{2j-2}$ but does not contain     $x_{2j-1}$, then  $R \in \mathcal{Z}_{j-1}$ and the induction hypothesis implies that   $L_i(R)$ eventually becomes  active   for every  integer $i \in \{1, \ldots, j-1\}$.   Combining these two observations and applying Lemma \ref{endpoints-prime}, we conclude that
$L_i(R)$ eventually becomes  active   for each integer $i \in \{0, 1, \ldots, j-1\}$ and each subset $R \in \mathcal{Z}_j$ containing    $x_{2j-2}$.
Further, for any  such integer $i$
and subset $R$, since
\(U_i(R)\) has two distinct active  neighbors \(L_i(R)\) and \(L_{i+1}(R)\),  we deduce   that  \(U_i(R)\)  eventually becomes active.

To apply Lemma \ref{5prime}, we make the substitutions  $P = D_{j}$, $Q = C_{j}$, $x = x_{2j-2}$, and $t = j$. We show that  both  conditions  (i) and (ii) of Lemma \ref{5prime} are satisfied. For every    $t$-subset $R\subseteq Q$,
the vertices \(L_{t}(P\cup R)\) and \(U_{t}(P\cup R)\) are active by the first stage  and by Lemmas \ref{endpoints} and \ref{endpoints-prime}. Also, Lemma \ref{endpoints} implies that the vertex
\(U_{2q}(P\cup R)=L_q(X\setminus(P\cup R))\)   is   active.   Moreover,    for each  integer $i \in \{0, 1, \ldots, t-1\}$  and  each   $t$-subset $R\subseteq Q$ containing     $x$,
the result established in the last line of the preceding paragraph guarantees that  $L_i(P\cup R)$ is   active.
These observations show that conditions (i) and (ii) of Lemma \ref{5prime} are satisfied.
It now follows from Lemma \ref{5prime}, together with Lemma \ref{endpoints},  that  $U_i(R)$ eventually becomes  active   for every  integer $i \in \{0, 1, \ldots, j-1\}$ and every  subset $R \in \mathcal{Z}_j$   that does not contain  $x_{2j-2}$.
In addition,   for any  such integer $i$
and subset $R$,     since   $U_{-1}(R)$ is  active   by Lemma \ref{endpoints-prime}  and since
\(L_i(R)\) has two distinct active neighbors \(U_{i-1}(R)\) and \(U_{i}(R)\), we conclude that \(L_i(R)\) eventually becomes active.
\renewcommand{\qedsymbol}{\rotatebox[origin=c]{45}{$\Box$}}
\end{proof}

\begin{lemma}\label{generalizedChDh}
For every integer  \(i\) and every subset    \(R\subseteq X\),   the   vertices   \(L_i(R)\) and \(U_i(R)\)    eventually become  active     from  \(\mathcal{F}\) under  the \(2\)-neighbor bootstrap percolation process on \(\mathbbmsl{O}_k\).
\end{lemma}

\begin{proof}[Proof of Lemma \ref{generalizedChDh}]
We begin by focusing   on  the  \(p\)-subsets of \(X\). Note   that
\(\mathcal{Z}_p\) coincides with the family of all \(p\)-subsets of \(X\).
For every integer $i\in\{q, \ldots, 2q \}$ and every $p$-subset    \(R\subseteq X\),     since   $L_i(R)=U_{i-q-1}(X \setminus R)$ and   $U_i(R)=L_{i-q}(X \setminus R)$, it follows from Lemmas   \ref{endpoints},    \ref{endpoints-prime}, and     \ref{ChDh}    that   \(L_i(R)\) and \(U_i(R)\)  both    eventually become  active. Combining this observation with Lemma \ref{ChDh}, we conclude that      \(L_i(R)\) and \(U_i(R)\)  both    eventually become  active    for all integers $i$ and all     $p$-subsets $R\subseteq X$.
We now consider an arbitrary integer \(i\) and distinguish between
the two cases \( |R|\leqslant p-1\) and \( |R|\geqslant p+1\).

First, suppose that  $|R|\leqslant p-1$.
We begin with the case
where \(R\) is a \((p-1)\)-subset of \(X\). Let
\(a_1,a_2\in X\setminus R\) be two distinct elements.    The vertex   \(L_{i}(R)\)  eventually becomes active, as  it  has two distinct active  neighbors   \(U_i(R\cup\{a_1\})\) and \(U_{i}(R\cup\{a_2\})\).
By     Lemma \ref{local}, the vertex
\(U_{i}(R)\) has two distinct active  neighbors, namely   \(L_{i}(R)\) and \(L_{i+1}(R)\), and hence  eventually    becomes active.
Iterating this argument, we  deduce that     $L_i(R)$ and $U_i(R)$ both     eventually become   active for all     subsets $R\subseteq   X$ with $|R|\leqslant p-1$.

Next, suppose that   $|R|\geqslant p+1$.
We begin with the case
where \(R\) is a \((p+1)\)-subset of \(X\). Let
\(b_1,b_2\in   R\) be two distinct elements.
The vertex    \(U_{i}(R)\) has      two distinct active  neighbors  \(L_i(R\setminus\{b_1\})\) and \(L_{i}(R\setminus\{b_2\})\).  Accordingly,    \(U_{i}(R)\)  eventually becomes active.  Further,
\(L_{i}(R)\) has two distinct active  neighbors, namely   \(U_{i}(R)\) and \(U_{i-1}(R)\), and thus  also   becomes active.
Iterating this argument,  we  conclude   that    $L_i(R)$ and $U_i(R)$ both   eventually become   active for all     subsets $R\subseteq   X$ with $|R|\geqslant p+1$.
\renewcommand{\qedsymbol}{\rotatebox[origin=c]{45}{$\Box$}}
\end{proof}

We are now ready to resume and complete the proof of Theorem   \ref{OGOG}.
To complete the proof, we demonstrate   that every $k$-subset of $M$ eventually becomes active. Fix a $k$-subset $R$ of $M$. We consider two cases, depending on the parity of  $k$.
First,  suppose   that $k$ is even. We apply Lemma    \ref{replacement} with   \(W=X\),  \(A=X\), and \(B=Y\). We claim that  $(W\setminus\{a\})\cup\{b\}$ eventually becomes active for any  elements $ a \in  A$
and $ b \in  B$.  Indeed, if   $b=y_i$ for some integer $ i  \in \{0, 1, \ldots, 2q\}$, then  the vertex  $(W\setminus\{a\})\cup\{b\}=L_i(X\setminus\{a\})$    eventually becomes active by  Lemma \ref{generalizedChDh}.
Now,  let    $S=X\setminus  R$  and
$T=R\cap Y$.
If \(S=\varnothing\), then \(R=X=L_0(X)\) and  hence   Lemma \ref {generalizedChDh} implies  that   \(R \)   eventually becomes  active. Thus,  we   assume that \(S\neq\varnothing\).
As  $|S|=|T|\geqslant1$, Lemma  \ref{replacement}   ensures that  the vertex $ R=  (W\setminus S)\cup T$
eventually becomes active.

Next, suppose that \(k\) is odd. We prove   by induction on
\(i\) that \(R\) eventually becomes active, where
\(i\in\{0,1,\ldots,2q\}\) is chosen so that
\(R\subseteq X\cup\{y_0, y_1, \ldots,y_i\}\). If
\(i=0\), then     $R=X\cup \{y_0\}=L_0(X)$   and hence  $R$ eventually becomes active   by Lemma  \ref{generalizedChDh}.
So, assume  that $i\geqslant1$ and that all $k$-subsets of $ X\cup\{y_0, y_1, \ldots, y_{i-1}\}$   eventually become active.
We apply Lemma    \ref{replacement} with   \(W=X\cup\{y_{i-1}\}\),  \(A=X\cup\{y_{i-1}\}\), and \(B=\{y_0, y_1,  \ldots, y_{i-2}, y_i\}\).
We claim that  $(W\setminus\{a\})\cup\{b\}$ eventually becomes active for any  elements $ a \in  A$
and $ b \in  B$.
Indeed, if   $b\in \{y_0, y_1, \ldots, y_{i-2}\}$,
then  $(W\setminus\{a\})\cup\{b\}$ is a $k$-subset of  $X\cup\{y_0, y_1, \ldots, y_{i-1}\}$ and hence  eventually becomes active by the induction hypothesis.
On the other hand, if  $b=y_i$, then  $(W\setminus\{a\})\cup\{b\}$  is either    $X\cup\{y_i\}=L_i(X)$ or $ (X\setminus\{a\})\cup\{y_{i-1}, y_i\}=L_{i-1}(X\setminus\{a\})$. In either case,  $(W\setminus\{a\})\cup\{b\}$  eventually becomes active using  Lemma \ref{generalizedChDh}. This establishes the claim.
Now,  let    $S=A\setminus  R$  and
$T=B\cap R$.
If \(S=\varnothing\), then  $R=X\cup \{y_{i-1}\}=L_{i-1}(X)$   and  hence   Lemma \ref {generalizedChDh} implies  that   \(R \)   eventually becomes  active. Thus,  we   assume that \(S\neq\varnothing\).
As  $|S|=|T|\geqslant1$,  Lemma  \ref{replacement}   guarantees  that    the vertex
$ R=  (W\setminus S)\cup T$
eventually becomes active. This completes the induction
and the proof of Theorem \ref{OGOG}.
\end{proof}

\noindent {\bf Declaration:} 
Theorems \ref{hh}, \ref{o-g-l-b}, and \ref{OGOG}, together with their proofs, were obtained with the assistance of the artificial intelligence tool \textsf{ChatGPT 6 Astra}, developed by \textsf{OpenAI}. The tool was also used to identify possible errors in the final version of the article. The authors wrote the entire article themselves, including all mathematical arguments, and take full responsibility for its entire content, including the accuracy of all mathematical statements and proofs.

\end{document}